\documentclass[english]{article}
\usepackage[T1]{fontenc}
\usepackage[latin9]{inputenc}
\usepackage{geometry}
\usepackage{color}
\usepackage{babel}
\usepackage{array}
\usepackage{float}
\usepackage{mathtools}
\usepackage{bm}
\usepackage{amsmath}
\usepackage{amsthm}
\usepackage{amssymb}
\usepackage{graphicx}
\usepackage[unicode=true,pdfusetitle,
 bookmarks=true,bookmarksnumbered=false,bookmarksopen=false,
 breaklinks=false,pdfborder={0 0 1},backref=false,colorlinks=false]
 {hyperref}

\makeatletter

\providecommand{\tabularnewline}{\\}
\floatstyle{ruled}
\newfloat{algorithm}{tbp}{loa}
\providecommand{\algorithmname}{Algorithm}
\floatname{algorithm}{\protect\algorithmname}
\usepackage{tikz}
\usetikzlibrary{calc}

\theoremstyle{plain}
\newtheorem{thm}{\protect\theoremname}
\theoremstyle{definition}
\newtheorem{defn}[thm]{\protect\definitionname}
\theoremstyle{plain}
\newtheorem{lem}[thm]{\protect\lemmaname}
\theoremstyle{plain}
\newtheorem{assumption}[thm]{\protect\assumptionname}
\theoremstyle{plain}
\newtheorem{prop}[thm]{\protect\propositionname}
\theoremstyle{remark}
\newtheorem{rem}[thm]{\protect\remarkname}
\theoremstyle{plain}
\newtheorem{cor}[thm]{\protect\corollaryname}

\usepackage{algorithmic}
\usepackage{wasysym}

\makeatother

\providecommand{\assumptionname}{Assumption}
\providecommand{\corollaryname}{Corollary}
\providecommand{\definitionname}{Definition}
\providecommand{\lemmaname}{Lemma}
\providecommand{\propositionname}{Proposition}
\providecommand{\remarkname}{Remark}
\providecommand{\theoremname}{Theorem}

\begin{document}
\global\long\def\conj{*}%

\global\long\def\Z{\mathbb{Z}}%

\global\long\def\R{\mathbb{R}}%

\global\long\def\C{\mathbb{C}}%

\global\long\def\H{{\cal H}}%

\global\long\def\X{{\cal X}}%

\global\long\def\Y{{\cal Y}}%

\global\long\def\e{{\mathbf{e}}}%

\global\long\def\et#1{{\e(#1)}}%

\global\long\def\ef{{\mathbf{\et{\cdot}}}}%

\global\long\def\x{{\mathbf{x}}}%

\global\long\def\w{{\mathbf{w}}}%

\global\long\def\xt#1{{\x(#1)}}%

\global\long\def\xf{{\mathbf{\xt{\cdot}}}}%

\global\long\def\d{{\mathbf{d}}}%

\global\long\def\b{{\mathbf{b}}}%

\global\long\def\u{{\mathbf{u}}}%

\global\long\def\y{{\mathbf{y}}}%

\global\long\def\yt#1{{\y(#1)}}%

\global\long\def\yf{{\mathbf{\yt{\cdot}}}}%

\global\long\def\z{{\mathbf{z}}}%

\global\long\def\v{{\mathbf{v}}}%

\global\long\def\h{{\mathbf{h}}}%

\global\long\def\s{{\mathbf{s}}}%

\global\long\def\c{{\mathbf{c}}}%

\global\long\def\p{{\mathbf{p}}}%

\global\long\def\f{{\mathbf{f}}}%

\global\long\def\g{{\mathbf{g}}}%

\global\long\def\a{{\mathbf{a}}}%

\global\long\def\rb{{\mathbf{r}}}%

\global\long\def\rt#1{{\rb(#1)}}%

\global\long\def\rf{{\mathbf{\rt{\cdot}}}}%

\global\long\def\mat#1{{\ensuremath{\bm{\mathrm{#1}}}}}%

\global\long\def\valpha{\mat{\alpha}}%

\global\long\def\vbeta{\mat{\beta}}%

\global\long\def\vtheta{\mat{\theta}}%

\global\long\def\veta{\mat{\eta}}%

\global\long\def\vmu{\mat{\mu}}%

\global\long\def\vrho{\mat{\rho}}%

\global\long\def\matN{\ensuremath{{\bm{\mathrm{N}}}}}%

\global\long\def\matA{\ensuremath{{\bm{\mathrm{A}}}}}%

\global\long\def\matB{\ensuremath{{\bm{\mathrm{B}}}}}%

\global\long\def\matC{\ensuremath{{\bm{\mathrm{C}}}}}%

\global\long\def\matD{\ensuremath{{\bm{\mathrm{D}}}}}%

\global\long\def\matP{\ensuremath{{\bm{\mathrm{P}}}}}%

\global\long\def\matU{\ensuremath{{\bm{\mathrm{U}}}}}%

\global\long\def\matV{\ensuremath{{\bm{\mathrm{V}}}}}%

\global\long\def\matM{\ensuremath{{\bm{\mathrm{M}}}}}%

\global\long\def\matR{\mat R}%

\global\long\def\matW{\mat W}%

\global\long\def\matK{\mat K}%

\global\long\def\matQ{\mat Q}%

\global\long\def\matS{\mat S}%

\global\long\def\matY{\mat Y}%

\global\long\def\matX{\mat X}%

\global\long\def\matI{\mat I}%

\global\long\def\matJ{\mat J}%

\global\long\def\matZ{\mat Z}%

\global\long\def\matL{\mat L}%

\global\long\def\matH{\mat H}%

\global\long\def\S#1{{\mathbb{S}_{N}[#1]}}%

\global\long\def\IS#1{{\mathbb{S}_{N}^{-1}[#1]}}%

\global\long\def\PN{\mathbb{P}_{N}}%

\global\long\def\TNormS#1{\|#1\|_{2}^{2}}%

\global\long\def\TNorm#1{\|#1\|_{2}}%

\global\long\def\InfNorm#1{\|#1\|_{\infty}}%

\global\long\def\FNorm#1{\|#1\|_{F}}%

\global\long\def\UNorm#1{\|#1\|_{\matU}}%

\global\long\def\UNormS#1{\|#1\|_{\matU}^{2}}%

\global\long\def\UINormS#1{\|#1\|_{\matU^{-1}}^{2}}%

\global\long\def\ANorm#1{\|#1\|_{\matA}}%

\global\long\def\BNorm#1{\|#1\|_{\mat B}}%

\global\long\def\HNorm#1{\|#1\|_{{\cal H}}}%

\global\long\def\HNormS#1{\|#1\|_{\H}^{2}}%

\global\long\def\XNormS#1#2{\|#1\|_{#2}^{2}}%

\global\long\def\AINormS#1{\|#1\|_{\matA^{-1}}^{2}}%

\global\long\def\BINormS#1{\|#1\|_{\matB^{-1}}^{2}}%

\global\long\def\BINorm#1{\|#1\|_{\matB^{-1}}}%

\global\long\def\ONorm#1#2{\|#1\|_{#2}}%

\global\long\def\T{\textsc{T}}%

\global\long\def\pinv{\textsc{+}}%

\global\long\def\Expect#1{{\mathbb{E}}\left[#1\right]}%

\global\long\def\ExpectC#1#2{{\mathbb{E}}_{#1}\left[#2\right]}%

\global\long\def\dotprod#1#2#3{(#1,#2)_{#3}}%

\global\long\def\dotprodsqr#1#2{(#1,#2)^{2}}%

\global\long\def\Trace#1{{\bf Tr}\left(#1\right)}%

\global\long\def\realpart#1{{\bf Re}\left(#1\right)}%

\global\long\def\nnz#1{{\bf nnz}\left(#1\right)}%

\global\long\def\range#1{{\bf range}\left(#1\right)}%

\global\long\def\nully#1{{\bf null}\left(#1\right)}%

\global\long\def\vol#1{{\bf vol}\left(#1\right)}%

\global\long\def\rank#1{{\bf rank}\left(#1\right)}%

\global\long\def\diag#1{{\bf diag}\left(#1\right)}%

\global\long\def\erfc#1{{\bf erfc}\left(#1\right)}%

\global\long\def\grad#1{{\bf grad}#1}%

\global\long\def\id#1{{\bf id_{#1}}}%

\global\long\def\st{\,\,\,\text{s.t.}\,\,\,}%

\title{Semi-Infinite Linear Regression and Its Applications}
\author{Paz Fink Shustin and Haim Avron}
\maketitle
\begin{abstract}
Finite linear least squares is one of the core problems of numerical
linear algebra, with countless applications across science and engineering.
Consequently, there is a rich and ongoing literature on algorithms
for solving linear least squares problems. In this paper, we explore
a variant in which the system's matrix has one infinite dimension
(i.e., it is a quasimatrix). We call such problems semi-infinite linear
regression problems. As we show, the semi-infinite case arises in
several applications, such as supervised learning and function approximation,
and allows for novel interpretations of existing algorithms. We explore
semi-infinite linear regression rigorously and algorithmically. To
that end, we give a formal framework for working with quasimatrices,
and generalize several algorithms designed for the finite problem
to the infinite case. Finally, we suggest the use of various sampling
methods for obtaining an approximate solution.
\end{abstract}

\section{Introduction}

Consider the classical linear least squares problem: given an $m\times n$
matrix $\matA$, and a vector $\b$, we seek to compute: 
\begin{equation}
\x^{\star}=\arg\min_{\x\in\R^{n}}\TNorm{\matA\x-\b}\,.\label{eq:finite-ls}
\end{equation}
The problem of solving Eq.~(\ref{eq:finite-ls}) is one of the most
fundamental problems of numerical linear algebra, and it has countless
applications throughout scientific computing and data science. As
such, there is a rich literature on algorithms for solving Eq.~(\ref{eq:finite-ls}).
In particular, there are algorithms that: compute an approximate solution
\cite{DMMS11}, compute a near exact solution \cite{RT08,AMT10,MSM14},
are designed for the over-determined case \cite{AMT10}, designed
for the under-determined case \cite{MSM14}, consider also the presence
of a regularizer \cite{ACW17,PW16}, and replace the two-norm with
some other norm \cite{ClarksonEtAl16}. The previous list is far from
exhaustive. Finding efficient algorithms for solving Eq.~(\ref{eq:finite-ls})
is an active research field.

In this paper, we explore a variant of Eq.~(\ref{eq:finite-ls})
in which $\matA$ is no longer a matrix, but a \emph{quasimatrix},
that is a matrix in which one of the two dimensions is infinite (while
the other dimension is finite). We call such problems \emph{'semi-infinite
linear regression'. }As we show, the semi-infinite case arises in
several applications, such as supervised learning and function approximation,
and allows for novel interpretations of existing algorithms. In contrast
to the rich literature on the finite (i.e., matrix) variant of Eq.~(\ref{eq:finite-ls}),
the semi-infinite case has been hardly treated in the literature (the
only exception we aware of is \cite{Trefethen10}).

The goal of this paper is to explore semi-infinite linear regression
rigorously and algorithmically. To that end, we first define the notion
of quasimatrix formally, and give the needed framework for working
with quasimatrices, both mathematically and algorithmically. The use
of the term 'quasimatrix' as a matrix which has columns or rows that
are functions first appears in the literature in \cite{Stewart98,Trefethen10},
but has so far been informal. Once we have the mathematical foundations,
we define semi-infinite regression formally, and discuss applications.

We then proceed to proposing algorithms for solving semi-infinite
linear regression problems. First, we discuss direct methods, which
factorize a quasimatrix $\mat A$ into a product of quasimatrices.
Even though most of the algorithms we present are straightforward
generalizations of classical methods for finite linear least squares
problems, we also show how in some cases the use of quasimatrix operations
can be sidestepped.

Next, we discuss iterative methods. It is possible to devise a wide
array of iterative methods for solving semi-infinite linear regression
by generalizing iterative methods for the finite case. However, for
conciseness we show a representative algorithm from each of the two
approaches: Krylov subspace methods and stochastic optimization. For
Krylov methods, we show how LSMR \cite{fong2011lsmr} can be generalized
to solve semi-infinite linear regression. For stochastic optimization,
we adapt a method based on stochastic variance reduce gradient descent
(SVRG) \cite{johnson2013accelerating,xiao2014proximal}. Unlike LSMR,
the adaptation of SVRG is less straightforward, and requires generalizing
it to work with integrable sums.

Motivated by recent literature on Randomized Numerical Linear Algebra
(RandNLA), we discuss algorithms that are based on \emph{sampling}
the semi-infinite linear regression problem. We discuss both randomized
sampling and deterministic sampling. For randomized sampling, we discuss
a sampling technique which we term as \emph{natural sampling}. Conceptually,
natural sampling is analogous to uniformly sampling rows or columns
from $\matA$ when dealing with finite linear least-squares problems,
though for semi-infinite linear regression problems, uniform sampling
is not always well defined. It is well known from the RandNLA literature
that it is better to sample based on the so-called \emph{leverage
scores}. For quasimatrices, the analogous operation is sampling using
the \emph{leverage function}, a generalization of leverage scores.
Similar results have been shown before for restricted cases \cite{avron2017random,cohen2017optimal,avron2019universal}.
We also propose a deterministic sampling method based on Gauss-Legendre
quadrature. Interestingly, this method does not have any finite analogue.

Most of the methods we present are based on existing algorithms for
the finite linear least squares case to some degree (the only exception
is the quadrature sampling). The main contribution of the paper is
in the rigorous and systematic treatment of the subject. We hope that
our systematic treatment of semi-infinite linear regression will spur
additional interest and research on this problem.

\section{\label{sec:quasimatrices}Quasimatrices}

The term 'quasimatrix' appears in the literature as a name for matrices
in which one of their dimensions is infinite \cite{Stewart98}. The
term was later adapted by the chebfun library~\textbf{\cite{battles2004extension}}\footnote{See \href{http://www.chebfun.org/docs/guide/guide06.html}{http://www.chebfun.org/docs/guide/guide06.html}.},
and a variety of papers related to that package, and other literature,
use it, e.g. \cite{Trefethen10,olver2008preconditioned,olver2010shifted,kandasamy2016learning,hashemi2018spectral}..
In previous literature, quasimatrices were defined and treated informally
as matrices which have columns or rows that are functions. For our
purposes, a more formal treatment is needed, and we provide it here.
Our approach is in many ways similar to the one taken by \cite{de1991alternative}
to defining fundamental notions such as rank and basis as ones derived
from linear maps in which the domain is finite dimensional vectors.

\paragraph{Notations and Basic Terminology.}

For an integer $n$, we denote $[n]=\{1,\dots,n\}$. Scalars are denoted
by lower-case Greek letters or by $x,y,\dots$. Given two Banach spaces,
${\cal X}$ and ${\cal Y}$, we denote by ${\cal B}(\X,\Y)$ the Banach
space of bounded linear operators from $\X$ to $\Y$. Given a Banach
space $\X$, $\X^{\conj}$ denotes the \emph{topological }dual space
of $\X$, i.e. the space of continuous linear functionals from $\X$
to $\R$ or $\C$. So, $\X^{\conj}={\cal B}(\X,\R)$ or $\X^{\conj}={\cal B}(\X,\C)$
(depending on the context). Vectors are denoted by $\x,\y,\dots$
and considered as column vectors (unless otherwise stated), and matrices
are denoted by $\matA,\mat B,\dots$ or upper-case Greek letters.
Quasimatrices are denoted by $\x,\y,\dots$ if they are lying in a
Hilbert space and otherwise by $\matA,\mat B,\dots$(defined later
in this paper). For a vector $\x$ or a matrix $\mat A$, the notation
$\x^{*}$ or $\mat A^{*}$ denotes the Hermitian transpose. The $n\times n$
identity matrix is denoted by $\mat I_{n}$. We use $\e_{1},\e_{2},\dots$
to denote the unit vectors and assume that their dimensions are clear
from the context. All vectors are considered as columns vectors, which
can be of finite or infinite dimension.

We use $L_{2}(\Omega,d\mu)$ to denote the space of complex-valued
square integrable functions over $\Omega$ with respect to the measure
$\mu$, i.e. the inner product in $L_{2}(\Omega,d\mu)$ is 
\[
\dotprod{\f}{\g}{L_{2}(\Omega,d\mu)}\coloneqq\int_{\Omega}\overline{\f(\veta)}\g(\veta)d\mu(\veta)\,.
\]

\subsection{Quasimatrix Algebra}

A matrix is a mapping from two indexes to a scalar. Alternatively,
a matrix can be viewed as mapping from a finite index set to finite
dimensional vectors, where the index set is either the column index
or the row index. For a quasimatrix we drop the condition that the
mapped vectors are finite dimensional, and instead require them to
be from a Hilbert space.
\begin{defn}
Let $n$ be a positive integer and let ${\cal H}$ be an Hilbert space
over $\R$ or $\C$. A \emph{tall quasimatrix} is a mapping from $[n]$
to $\H$. A \emph{wide quasimatrix }is a mapping from $[n]$ to $\H^{\conj}$.
\end{defn}

We generally omit the adjectives tall and wide when the text refers
to both types, or when the relevant type is clear from the context.
We say the size of a tall quasimatrix is $m\times n$ if $m$ is the
dimension of the Hilbert space $\H$. We generally write $\infty\times n$
if $\H$ has infinite dimension. A similar notion of size applies
for wide quasimatrices. To avoid clutter, henceforth we assume that
$\H$ is defined over $\C$, and leave for the reader to deduce how
some of the description is somewhat simplified for the real case.

For a tall quasimatrix $\matA$, we refer to the values at the various
indexes as the \emph{columns} of the quasimatrix. We use the following
notation
\[
\matA=\left[\begin{array}{ccc}
\a_{1} & \cdots & \a_{n}\end{array}\right]
\]
to denote the tall quasimatrix $\matA$ which maps $j\in[n]$ to $\a_{j}$,
where $\a_{1},\dots,\a_{n}\in\H$.

Let $\b_{1},\dots,\b_{m}\in\H$, and $\b_{1}^{\conj},\dots,\b_{m}^{\conj}\in\H^{*}$
their adjoints. We denote

\[
\matB=\left[\begin{array}{c}
\b_{1}^{\conj}\\
\vdots\\
\b_{m}^{\conj}
\end{array}\right]
\]
for the wide quasimatrix $\matB$ which maps $j\in[m]$ to $\b_{j}^{\conj}$.
If a wide quasimatrix $\matB$ maps $j$ to $\x\in\H^{\conj}$, the
Riesz Representation Theorem implies that there exists a $\b\in\H$
such that $\x=\b^{\conj}$, so every wide quasimatrix can be written
in this way. For a wide quasimatrix $\matB$, we refer to the adjoints
of the values at the indexes as the \emph{rows} of the quasimatrix.
Note that both the columns of a tall quasimatrix, and the rows of
a wide quasimatrix, are vectors in $\H$.

In the rest of Section~\ref{sec:quasimatrices}, $\matA$ is a tall
quasimatrix with columns $\a_{1},\dots,\a_{n}$, and $\matB$ is a
wide quasimatrix with rows $\b_{1}^{\conj},\dots,\b_{m}^{\conj}$.

The conjugate transpose of a tall quasimatrix $\matA$ is the wide
quasimatrix $\matB$ whose coordinates are the adjoints of the corresponding
coordinates of $\matA$. The conjugate transpose of a wide quasimatrix
$\matB$ is the tall quasimatrix $\matA$ whose coordinates are the
adjoints of the corresponding coordinates of $\matB$, which corresponds
to removing the adjoints. These definitions are consistent with the
notations above.

Given a tall quasimatrix $\matA$ and a vector $\x\in\C^{n}$, we
define the product of $\matA$ and $\x$ as $\matA\x=\sum_{j=1}^{n}x_{j}\a_{j}$.
This definition naturally extends to the definition of a product $\matA\matX$,
where $\matX\in\C^{n\times k}$, as the tall quasimatrix whose columns
are $\matA\x_{1},\dots\matA\x_{k}$ (where $\x_{1},\dots,\x_{n}$
are the columns of $\matX$). Given a wide quasimatrix $\matB$ and
a $\x\in\H$ we define 
\[
\matB\x=\left[\begin{array}{c}
\b_{1}^{\conj}\x\\
\vdots\\
\b_{m}^{\conj}\x
\end{array}\right]\,.
\]
This definition naturally extends to the definition of a product of
$\matB$ and $\matA$ as the $m\times n$ matrix whose columns are
$\matB\a_{1},\dots\matB\a_{n}$.

We now define the left product. Given a tall quasimatrix $\matA$
and a vector $\x\in\H$, we define $\x^{\conj}\matA=(\matA^{\conj}\x)^{\conj}$.
This definition is consistent with viewing $\x^{\conj}$ as a $1\times\infty$
quasimatrix, and the previous definition of $\matB\matA$. Similarly,
$\x^{\conj}\matB=(\matB^{\conj}\x)^{\conj}$ for $\x\in\R^{m}$. These
definitions naturally extend to the left product of a matrix and a
quasimatrix. The product algebra we have defined over matrices and
quasimatrices is associative (but, of course, not commutative).

It is well known that a $m\times n$ complex matrix is, in fact, a
coordinate representation of a linear transformation from $\C^{n}$
to $\C^{m}$ under the standard basis, and that choosing a different
basis leads to a different matrix (coordinate) representation. Similar
to finite matrices, quasimatrices define bounded linear transformations
between finite-dimensional Euclidean spaces and $\H$. Concretely,
for a tall quasimatrix $\matA$, we can define the transformation
${\cal A}:\x\in\C^{n}\mapsto{\cal \matA}\x\in\H$. Conversely, given
a bounded linear transformation ${\cal A}:\C^{n}\to{\cal H}$, for
the tall quasimatrix $\matA=\left[{\cal A}\e_{1}\cdots{\cal A\e}_{n}\right]$,
the corresponding linear transformation is ${\cal A}$. Thus, we can
abuse notation and use $\matA$ to denote both the quasimatrix and
the linear transformation it defines. Likewise, every wide quasimatrix
$\matB$ defines a bounded linear transformation $\matB:\x\in\H\mapsto\matB\x\in\R^{m}$,
likewise abusing notation (this is well justified by the Riesz Representation
Theorem). Taking the conjugate transpose of a tall or wide quasimatrices
produces a quasimatrix which represents the adjoint of the transformation
defined by the original quasimatrix, thus our notation is consistent
with that operation as well.

Similarly to the finite dimensional matrix case, the extended product
algebra over matrices and quasimatrices is consistent with composition
in the transformation spaces. That is, given two matrices or quasimatrices
$\matX$ and $\matY$, with sizes or underlying Hilbert space such
that the product $\matX\matY$ is defined, the transformation defined
by $\matX\matY$ (which is a matrix or a quasimatrix) is exactly the
same as the transformation obtained by $\matX$ composed with $\matY$.
However, we remark that if $\matX$ is a tall quasimatrix and $\matY$
is a wide quasimatrix, then we can certainly define the transformation
from $\H$ to $\H$ by composing $\matX$ with $\matY$, but that
transformation is not represented by a quasimatrix.

We have defined columns of a tall quasimatrix and rows of a wide quasimatrix.
Defining the rows of a tall quasimatrix or the columns of a wide quasimatrix
is less straightforward, and in some senses impossible. Intuitively,
if $\H$ is a space of functions over some index set $\Omega\subseteq\R^{n}$,
then row $\veta$ of a tall quasimatrix is simply the evaluation of
the columns at a location $\veta\in\Omega$ (and likewise for wide
quasimatrices). However, requiring ${\cal H}$ to be a space of functions
is somewhat restrictive. In particular, note that $L_{2}$ spaces
are, strictly speaking, spaces of equivalence classes of functions,
and point-wise evaluation is not really well-defined.

However, note that if ${\cal H}$ is a reproducing kernel Hilbert
space (RKHS), then we can define the notion of rows (or columns) of
a tall (wide) quasimatrix in a way that is consistent with the use
of identity vectors in finite matrices. If $\H$ is a RKHS over $\R^{d}$
then for every $\veta\in\R^{d}$ the point-wise evaluation $\f\in\H\mapsto\f(\veta)$
is a bounded linear transformation. Thus, there is a unique $\e_{\veta}\in\H$
such that for every $\f\in\H$ we have $\e_{\veta}^{\conj}\f=\f(\veta)$.
Thus, we define row $\veta$ of a $\infty\times n$ quasimatrix $\matA$
over a RKHS $\H$ as $\e_{\veta}^{\conj}\matA\in\R^{1\times n}$.
For a $m\times\infty$ quasimatrix $\matB$, column $\veta$ is defined
as $\matB\e_{\veta}$.

 Many notions related to matrices can be easily generalized to quasimatrices.
For example, the rank is the dimension of space spanned by the columns
(rows) of a tall (wide) quasimatrix. Obviously, the rank cannot be
larger then the size of the finite dimension, and similar properties
of matrix rank can be shown for quasimatrix rank.

\subsection{Numerical Computing with Quasimatrices}

In subsequent sections, we describe algorithms that ``operate''
on quasimatrices and functions. Such algorithms assume a model of
computation in which functions are primitive types, and certain operations
between functions are allowed (e.g., taking the integral of a function).
Of course, such computations are not supported in hardware by general
purpose computing machines. However, the software package \textbf{chebfun}\footnote{\href{http://www.chebfun.org/}{http://www.chebfun.org/}}\textbf{
}does provide this abstraction in software~\cite{battles2004extension}.
Thus, we refer to this model of computation as the \emph{chebfun model}.

In numerical computing, it is customary to regard floating-point operations
(FLOPs) as the costly operations, and thus runtime analysis focuses
on counting FLOPs. In the chebfun model, arguably the costly operations
are operations on functions. Thus, when analyzing algorithms in the
chebfun model we count FUNction OPerations (FUNOPs).

Specifically, we assume the following operations are supported, each
costing one FUNOP: multiplying a function by a scalar, adding or subtracting
two functions, evaluating a function at a point, and taking the inner
product of two functions,

Of course, wherever possible we attempt to describe algorithms that
operate under the standard model of computation (no FUNOPs). Such
algorithms usually require additional assumptions on the quasimatrices
involved.

\subsection{Coordinate Representation of Quasimatrices over $L_{2}$ Spaces}

As explained in the previous sections, the rows of a of a tall quasimatrix
or the columns of a quasimatrix cannot be defined for quasimatrices
over $L_{2}$ spaces. For most algorithms we describe that use the
chebfun model this is not an issue. However, when we discuss algorithms
that perform sampling and operate in the standard model, we need access
to rows/columns so they can be sampled. In such cases we need to assume
that the algorithm, when applied to quasimatrices over $L_{2}$, has
additional information in the form of a \emph{coordinate representation
}of the quasimatrix it operates on.
\begin{defn}
Suppose $\matA$ is a quasimatrix over $L_{2}(\Omega,d\mu)$ whose
finite dimension is $n$. A coordinate representation of $\matA$
is a function $\z:\Omega\to\C^{n}$ such that 
\[
\int_{\Omega}\TNormS{\z(\veta)}d\mu(\veta)<\infty
\]
and:
\begin{enumerate}
\item If $\matA$ is tall, for every $\x\in\C^{n}$,\textbf{
\[
\matA\x=\sum_{i=1}^{n}x_{i}\overline{\z(\cdot)_{i}}
\]
}where the above equality should be interpreted in the $L_{2}(\Omega,d\mu)$
sense and $\z(\veta)_{i}$ is coordinate $i$ of $\z(\veta)$ (for
$\veta\in\Omega$).
\item If $\matA$ is wide, for every $\x\in L_{2}(\Omega,d\mu)$ and $j\in[n]$,
\[
\e_{j}^{*}\matA\x=\dotprod{\overline{\z(\cdot)_{j}}}{\x}{L_{2}(\Omega,d\mu)}.
\]
\end{enumerate}
\end{defn}

The definition implies that if $\z$ is a coordinate representation
of $\matA$ then it is also a coordinate representation of $\matA^{\conj}$.
Essentially, for a tall quasimatrix with a coordinate representation
$\z$, column $i$ is $\overline{\z(\cdot)_{i}}$, and for a wide
quasimatrix with a coordinate representation $\z$, row $i$ is $\overline{\z(\cdot)_{i}}^{\conj}$.
If the quasimatrix is defined over $L_{2}(\Omega,d\mu)$, we say that
$\Omega$ is the \emph{index set }of the infinite dimension. We now
say that for an index $\veta\in\Omega$, row $\veta$ of a tall quasimatrix
with coordinate representation $\z$ is $\z(\veta)^{\conj}$, and
column $\veta$ of a wide quasimatrix with coordinate representation
$\z$ is $\z(\veta)$.

Note that the definition also implies the following. For $\matA$
we have,
\[
\matA^{*}\matA=\left[\begin{array}{ccc}
\dotprod{\overline{\z_{\matA}(\cdot)_{1}}}{\z_{\matA}(\cdot)_{1}}{L_{2}(\Omega,d\mu)} & \dots & \dotprod{\overline{\z_{\matA}(\cdot)_{1}}}{\z_{\matA}(\cdot)_{n}}{L_{2}(\Omega,d\mu)}\\
\vdots & \ldots & \vdots\\
\dotprod{\overline{\z_{\matA}(\cdot)_{n}}}{\z_{\matA}(\cdot)_{1}}{L_{2}(\Omega,d\mu)} & \dots & \dotprod{\overline{\z_{\matA}(\cdot)_{n}}}{\z_{\matA}(\cdot)_{n}}{L_{2}(\Omega,d\mu)}
\end{array}\right]=\int_{\Omega}\z_{\matA}(\veta)\z_{\matA}(\veta)^{*}d\mu(\veta)
\]
and similarly for the product $\matB\matB^{*}$.

\subsection{Quasimatrix Factorizations}

Matrix factorizations such as QR and SVD are used to define direct
methods for solving linear regression problems (and more generally,
in matrix analysis at large). Thus, it is no surprising that they
can be used to solve semi-infinite linear regression problems as well,
as was already noted in~\cite{Trefethen10}. Various quasimatrix
factorizations are already mentioned in~\cite{trefethen1997numerical,battles2004extension,Trefethen10},
and are further developed in \cite{townsend2015continuous}. They
can be formulated in our formal quasimatrix framework (previous aforementioned
works used quasimatrices in an informal manner). In Table\ref{tab:quasi_factorizations}
we detail a few key quasimatrix factorizations of a tall quasimatrix
$\matA$. Factorizations for a wide quasimatrix $\matB$ can be obtained
by taking the conjugate transpose of a factorization of $\matB^{\conj}$.
We also detail in Table~\ref{tab:quasi_factorizations} the FUNOPs
cost of forming the various quasimatrix factorization.

\begin{table}[t]
\caption{\label{tab:quasi_factorizations}Factorizations of a tall Quasimatrix
$\protect\matA$ with $n$ columns over ${\cal H}$.}

\centering{}%
\begin{tabular}{|>{\centering}p{1.5cm}|>{\centering}p{5.5cm}|>{\centering}p{3cm}|>{\centering}p{2.5cm}|}
\hline 
 & {\small{}Factorization Form} & {\small{}Reference} & {\small{}FUNOPs in the chebfun model}\tabularnewline
\hline 
\hline 
{\scriptsize{}Reduce QR using Gram-Schmidt} & {\scriptsize{}$\matA=\matQ\matR$}{\scriptsize\par}

{\scriptsize{}$\matQ\in\R^{\infty\times n},\,\matQ^{*}\matQ=\matI_{n}$}{\scriptsize\par}

{\scriptsize{}$\matR\in\mathbb{R}^{n\times n}$ upper diagonal}{\scriptsize\par}

{\scriptsize{}$R_{ij}=\begin{cases}
\dotprod{\a_{i}}{\a_{j}}{\H} & j\geq i\\
0 & j<i
\end{cases}\,.$} & - & {\small{}$n(n+1)$}\tabularnewline
\hline 
{\scriptsize{}Reduce QR using Householder Triangulation} & {\scriptsize{}$\matA=\matQ\matR$}{\scriptsize\par}

{\scriptsize{}$\matQ=\matH_{1}\matH_{2}\cdots\matH_{n}\mat E\matS$}{\scriptsize\par}

{\scriptsize{}$\matH_{1},\dots,\matH_{n}\in{\cal B}(\H,\H)$ Householder
reflectors}{\scriptsize\par}

{\scriptsize{}$\mat E=\left[\begin{array}{cccc}
\e_{1}^{{\cal H}} & \e_{2}^{{\cal H}} & \dots & \e_{n}^{{\cal H}}\end{array}\right]\in\R^{\infty\times n}$}{\scriptsize\par}

{\scriptsize{}$\e_{1}^{{\cal H}},\e_{2}^{{\cal H}},\dots$ predetermined
sequence of orthonormal vectors in $\H$.}{\scriptsize\par}

{\scriptsize{}$\matS\in\mathbb{R}^{n\times n}$ diagonal sign matrix}{\scriptsize\par}

{\scriptsize{}$\matR\in\mathbb{R}^{n\times n}$ upper diagonal}{\scriptsize\par}

{\scriptsize{}$R_{ij}=\dotprod{\e_{i}^{{\cal H}}}{\a_{j}}{\H}$} & {\small{}\cite{Trefethen10}} & {\small{}$3n(3n-1)/2+6n$}\tabularnewline
\hline 
{\scriptsize{}SVD} & {\scriptsize{}$\matA=\matU\Sigma\matV^{*}$}{\scriptsize\par}

{\scriptsize{}$\matU\in\R^{\infty\times n},\,\matU^{*}\matU=\matI_{n}$}{\scriptsize\par}

{\scriptsize{}$\matV\in\mathbb{R}^{n\times n},\,\matV^{*}\matV=\matI_{n}$}\\
{\scriptsize{}$\Sigma\in\mathbb{R}^{n\times n}$ non-negative diagonal
matrix} & {\small{}\cite{battles2004extension,battles2005numerical,koyano2017infinite,kamel2019visual,kandasamy2016learning,Trefethen10}} & QR {\small{}cost} + $(2n-1)n$\tabularnewline
\hline 
\end{tabular}
\end{table}

Using the SVD factorization, we define the condition number of a quasi-matrix
to be $\kappa(\matA)\coloneqq\sigma_{1}/\sigma_{n}$ where $\Sigma=\diag{\sigma_{1},\dots,\sigma_{n}}$
in the SVD factorization.

\section{\label{sec:SILR}Semi-Infinite Linear Regression: Problem Statement
and Examples}

In this paper, we are mainly concerned with the solution of regularized
linear least squares regression problems with quasimatrices. We specifically
consider ridge regularization (also called Tikhonov regularization).
We call such problems \emph{Semi-Infinite Linear Regression (SILR)}.
Both the overdetermined case and the underdetermined case are considered.
In the overdetermined case, we are given a $\infty\times n$ quasimatrix
$\matA$ over $\H$, a target $\b\in\H$, and a regularization parameter
$\lambda\geq0$. The goal of SILR is to find $\x\in\C^{n}$ such that
\begin{equation}
\HNormS{\matA\x-\b}+\lambda\TNormS{\x}=minimum.\label{eq:over-silr}
\end{equation}
In the underdetermined case, we are given a $n\times\infty$ quasimatrix
$\matA$ over $\H$, a target $\b\in\C^{n}$, and a regularization
parameter $\lambda\geq0$. Our goal is to find a $\x\in\H$ such that
\begin{equation}
\TNormS{\matA\x-\b}+\lambda\HNormS{\x}=minimum.\label{eq:under-silr}
\end{equation}
 For simplicity, in both cases we either assume that $\matA$ has
full rank or that $\lambda>0$. This makes the solution unique, and
we always denote it by $\x^{\star}$.

We now give examples in which SILR is involved. We focus on cases
where SILR is solved approximately by sampling the quasimatrix in
order to turn the problem in regular finite linear regression problem.

\subsection{\label{subsec:LS_approx}Least Squares Approximation of a Function}

Suppose we are given a function $\f\in\H=L_{2}([-1,1],d\lambda)$
(or any other Hilbert space), and a finite dimensional subspace ${\cal V}$
of $\H$ (e.g., the space of polynomials up to a certain degree).
We want to find the optimal approximation (in the $\H$ sense) of
$\f$ in ${\cal V}$, which we denote by $f_{{\cal V}}$. Denote by
$n$ the dimension of ${\cal V}$, and let $\v_{1},\dots,\v_{n}$
be a basis for ${\cal V}$. Define the $\infty\times n$ quasimatrix
$\matA=\left[\begin{array}{ccc}
\v_{1} & \cdots & \v_{n}\end{array}\right]$. Then, $f_{{\cal V}}=\mat A\x^{\star}$ where 
\begin{equation}
\x^{\star}=\arg\min_{\x\in\mathbb{R}^{n}}\XNormS{\mat A\x-\f}{\H}\,.\label{eq:approx_fun-1}
\end{equation}

A closely related problem is the problem of reconstructing an unknown
function $f$ on a domain $\X$ from samples at randomly chosen points
\cite{cohen2013stability}. In this problem setting we are given $y_{i}=f(\x_{i})+\epsilon_{i}$
at $m$ given data points $\x_{1},\dots,\x_{m}$ sampled i.i.d from
some distribution $\rho$ on $\X$ (we do not assume we have an explicit
formula for $\rho,$or that we can produce additional samples; we
only assume such a distribution exists). The scalars $\epsilon_{1},\dots,\epsilon_{n}$
are noise terms, which might be zero in the noiseless case. We can
connect this problem to Eq.~(\ref{eq:approx_fun-1}) in the following
way, originally discussed in~\cite{cohen2013stability}. We setup
a finite dimensional subspace ${\cal V}$ and try to approximate $f_{{\cal V}}$
via sampling. Specifically, let $\mat A_{s}\in\mathbb{R}^{m\times n}$
be a ``rows sample'' of the quasimatrix $\matA$, i.e., defined
by $(\matA_{s})_{ij}=\v_{j}(\x_{i})$ , and let
\begin{equation}
\tilde{\x}=\arg\min_{\x\in\mathbb{R}^{n}}\XNormS{\matA_{s}\x-\y}2\,.\label{eq:sample-approx-fun}
\end{equation}
The approximation is $\tilde{f_{{\cal V}}}=\matA\tilde{\x}$. In \cite{cohen2013stability}
the authors provide a criterion on $s$ that describes the needed
amount of samples to ensure that the least squares method is stable
and that its accuracy is comparable to the best approximation error
of $\f$ by elements from ${\cal V}$. Note that Eq.~(\ref{eq:sample-approx-fun})
is a sampled version of Eq.~(\ref{eq:approx_fun-1}). We discuss
solving SILR problems using sampling in Section~\ref{sec:sampling}.

\subsection{\label{subsec:krr}Kernel Ridge Regression}

Kernel ridge regression is an important method for supervised learning.
Recall the problem of supervised learning: given training data $(\x_{1},y_{1}),\dots,(\x_{n},y_{n})\in\X\times\Y$,
where $\X\subseteq\R^{d}$ is an input domain and $\Y\subseteq\R$
is an output domain, we wish to infer some functional dependency between
the outputs and the inputs \cite{Cucker02onthe}. In \emph{kernel
ridge regression, }one starts with a positive definite kernel function
$k:\X\times\X\to\R$. The kernel is associated with a reproducing
kernel Hilbert space (RKHS) $\H_{k}$ which is the completion of the
function space
\[
\left\{ \sum_{i=1}^{m}\alpha_{i}k(\x_{i},\cdot)\,|\,\x_{i}\in\X,\alpha_{i}\in\R,m\in\Z_{+}\right\} 
\]
equipped with the inner product 
\[
\dotprod{\sum_{i=1}^{m}\alpha_{i}k(\x_{i},\cdot)}{\sum_{j=1}^{n}\beta_{j}k(\x_{j},\cdot)}{{\cal H}_{k}}=\sum_{i=1}^{m}\sum_{j=1}^{n}\alpha_{i}\beta_{j}k(\x_{i},\x_{j})\,.
\]
For some $\lambda>0$, the kernel ridge regression estimator is 
\begin{equation}
\f^{\star}=\arg\min_{\f\in\H_{k}}\sum_{i=1}^{n}(\f(\x_{i})-y_{i})^{2}+\lambda\XNormS{\f}{\H_{k}}\,.\label{eq:krr}
\end{equation}
The celebrated Representer Theorem \cite{SHB01} guarantees that $\f^{\star}$
can be written as 
\begin{equation}
\f^{\star}(\x)=\sum_{i=1}^{n}\alpha_{i}^{\star}k(\x_{i},\x)\label{eq:representer-expansion}
\end{equation}
for some $\alpha_{1}^{\star},\dots,\alpha_{n}^{\star}\in\R$ (note
that $k(\x_{i},\cdot)\in\H_{k}$ so $\sum_{i=1}^{n}\alpha_{i}^{\star}k(\x_{i},\cdot)\in\H_{k}$).
Simple linear algebra now implies that we can find $\alpha_{1},\dots,\alpha_{n}$
by solving the linear system 
\begin{equation}
(\matK+\lambda\matI_{n})\valpha=\y\label{eq:krr-linear-eq}
\end{equation}
where $\matK\in\R^{n\times n}$ is the matrix defined by $K_{ij}=k(\x_{i},\x_{j}$)
and $\y=\left[y_{1}\cdots y_{n}\right]^{\T}\in\R^{n}$.

\subsubsection{KRR as Semi-Infinite Linear Regression}

We now show how Eq.~(\ref{eq:krr}) can be written as a SILR problem.
Define the $n\times\infty$ quasimatrix $\matA$ over $\H_{k}$:

\[
\matA=\left[\begin{array}{c}
k(\cdot,\x_{1})^{\conj}\\
\vdots\\
k(\cdot,\x_{n})^{\conj}
\end{array}\right]\,.
\]
Due to the reproducing property of RKHS, $\dotprod{\f}{k(\cdot,\x_{j})}{{\cal H}_{k}}=\f(\x_{j})=(\matA\f)_{j}$
and we have 
\begin{equation}
\f^{\star}=\arg\min_{\f\in\H_{k}}\TNormS{\matA\f-\y}+\lambda\XNormS{\f}{\H_{k}}\,.\label{eq:krr-rkhs}
\end{equation}
Thus, the kernel ridge regression estimator is the solution to an
underdetermined SILR problem. In fact, using Eq.~(\ref{eq:krr-linear-eq})
to solve Eq.~(\ref{eq:krr-rkhs}) is an instance of a direct method
for solving underdetermined SILR problems; see Section~\ref{subsec:direct-under}.

In Eq.~(\ref{eq:krr-rkhs}), the quasimatrix $\matA$ is defined
over a RKHS. In certain cases, the problem can be cast as a SILR problem
with quasimatrices defined over a $L_{2}$ space, and this leads to
approximation methods based on sampling. The following is based on
the seminal work of Rahimi and Recht on random Fourier features \cite{rahimi2008random}.
Suppose that $k$ is a shift-invariant positive definite function,
that is $k(\x,\z)=k(\x-\z)$ for some positive definite $k(\cdot)$
(note that we abuse notation in denoting by $k$ both the kernel and
the positive definite function that defines it). Further assume that
$k$ is normalized in the sense that $k(\x,\x)=1$. According to Bochner's
Theorem, there exists a probability measure $\mu$ such that 
\[
k(\x,\z)=k(\x-\z)=\int_{\R^{d}}e^{-2\pi i(\x-\z)^{\T}\veta}d\mu(\veta)\,.
\]
Define the function $\varphi:\X\times\R^{d}\to\C$:
\[
\varphi(\x,\veta)=e^{2\pi i\x^{\T}\veta}\,.
\]
For fixed $\x,\z\in\X$ we have 
\[
\dotprod{\varphi(\x,\cdot)}{\varphi(\z,\cdot)}{L_{2}(\R^{d},d\mu)}=\int_{\R^{d}}e^{-2\pi i(\x-\z)^{\T}\veta}d\mu(\veta)=\int_{\R^{d}}e^{-2\pi i(\x-\z)^{\T}\veta}p(\veta)d\veta=k(\x,\z)
\]
so $\varphi(\x,\cdot)\in L_{2}(\R^{d},d\mu)$ for every $\x\in\X$.
Let us now define the $n\times\infty$ quasimatrix $\matB$ over $L_{2}(\R^{d},d\mu)$:

\begin{equation}
\matB=\left[\begin{array}{c}
\varphi(\x_{1},\cdot)^{\conj}\\
\vdots\\
\varphi(\x_{n},\cdot)^{\conj}
\end{array}\right]\,.\label{eq:quasi_RFF}
\end{equation}

\begin{lem}
Assuming that $\matK$ is full rank or $\lambda>0$, the following
holds: 
\[
\f^{\star}(\x)=\dotprod{\varphi(\x,\cdot)}{\w^{\star}}{L_{2}(\R^{d},d\mu)}
\]
where 
\begin{equation}
\w^{\star}=\arg\min_{\w\in L_{2}(\R^{d},d\mu)}\TNormS{\matB\w-\y}+\lambda\XNormS{\w}{L_{2}(\R^{d},d\mu)}\,.\label{eq:l2_krr}
\end{equation}
\end{lem}

\begin{proof}
Let 
\[
\w^{\star}=\arg\min_{\w\in L_{2}(\R^{d},d\mu)}\TNormS{\matB\w-\y}+\lambda\XNormS{\w}{L_{2}(\R^{d},d\mu)}\,.
\]
Since $\range{\mat B^{*}}$ is a a closed linear subspace of ${\cal H}$,
there exists $\mat{\v}^{\star}\in\mathbb{R}^{n}$ such that $\w^{\star}=\mat B^{*}\mat{\v}^{\star}+\z$
where $\z\perp\range{\mat B^{*}}.$ Since $\matB$, viewed as an operator,
is bounded, $\nully{\matB}=\left(\range{\mat B^{*}}\right)^{\perp}$,
so $\mat B\z=0$. Now, since $\z\neq0$ can only increase $\lambda\XNormS{\w}{L_{2}(\R^{d},d\mu)}$
we conclude that $\z=0$. Thus, $\w^{\star}=\mat B^{*}\mat{\v}^{\star}$
and we can write
\[
\min_{\w\in L_{2}(\R^{d},d\mu)}\TNormS{\matB\w-\y}+\lambda\XNormS{\w}{L_{2}(\R^{d},d\mu)}=\min_{\v\in\mathbb{R}^{n}}\TNormS{\matB\mat B^{*}\v-\y}+\lambda\XNormS{\mat B^{*}\v}{L_{2}(\R^{d},d\mu)}=\min_{\v\in\mathbb{R}^{n}}\TNormS{\matK\v-\y}+\lambda\v^{\T}\matK\v
\]
where $\matK=\matB\mat B^{*}\in\mathbb{R}^{n\times n}$ is the kernel
matrix previously defined. The optimal solution is $\v^{\star}=(\matK+\lambda\matI_{n})^{-1}\y$,
i.e., $\v^{\star}=\valpha^{\star}$, so $\w^{\star}=\sum_{j=1}^{n}\alpha_{j}^{\star}\varphi(\x_{j},\cdot)$.
We now have
\begin{align*}
\dotprod{\varphi(\x,\cdot)}{\w^{\star}}{L_{2}(\R^{d},d\mu)} & =\int_{\R^{d}}\overline{\varphi(\x,\veta)}\w^{\star}(\veta)d\mu(\veta)\\
 & =\int_{\R^{d}}e^{-2\pi i\x^{\T}\veta}\left(\sum_{j=1}^{n}\alpha_{j}^{\star}e^{2\pi i\x_{j}^{\T}\veta}\right)d\mu(\veta)\\
 & =\sum_{j=1}^{n}\alpha_{j}^{\star}\int_{\R^{d}}e^{-2\pi i(\x-\x_{j})^{\T}\veta}d\mu(\veta)\\
 & =\sum_{j=1}^{n}\alpha_{j}^{\star}k(\x,\x_{j})=\f^{\star}(\x)\,.
\end{align*}
\end{proof}
The quasimatrix $\matB$ is over complex-valued $L_{2}$ spaces. It
is possible to actually define an equivalent SILR problem with a quasimatrix
over a real-valued $L_{2}$ space. Let $\hat{\Omega}=\R^{d}\times[0,2\pi]$
and $\hat{\mu}=\mu\times U(0,2\pi)$ where $U(0,2\pi$) is the uniform
measure on $[0,2\pi${]}. Now, let $L_{2}(\hat{\Omega},d\hat{\mu}$)
denote the space of \emph{real-valued} square integrable functions
with respect to the measure $\hat{\mu}$. Define the function $\hat{\varphi}:{\cal X}\times\hat{\Omega}\to\R$:
\[
\hat{\varphi}(\x,(\veta,b))=\sqrt{2}\cos(\x^{\T}\veta+b)\,.
\]
Now, let
\[
\matC=\left[\begin{array}{c}
\hat{\varphi}(\x_{1},\cdot)^{\conj}\\
\vdots\\
\hat{\varphi}(\x_{n},\cdot)^{\conj}
\end{array}\right]\,.
\]
Then, 
\[
\f^{\star}(\x)=\dotprod{\hat{\varphi}(\x,\cdot)}{\u^{\star}}{L_{2}(\hat{\Omega},d\hat{\mu})}
\]
where 
\[
\u^{\star}=\arg\min_{\u\in L_{2}(\hat{\Omega},d\hat{\mu})}\TNormS{\matC\u-\y}+\lambda\XNormS{\u}{L_{2}(\hat{\Omega},d\hat{\mu})}\,.
\]
See~\cite{rahimi2008random}.

\subsubsection{\label{subsec:KRR_samp}Approximating KRR using Quasimatrix Sampling}

Computing the exact KRR estimator is costly (since $\matK$ is typically
dense, finding $\mat{\alpha}$ in Eq.~(\ref{eq:krr-linear-eq}) costs
$O(n^{3})$ using direct methods; computing $\f^{\star}(\x)$ for
some $\x$ using Eq.~(\ref{eq:representer-expansion}) costs $O(nd)$;
since computing $\f^{\star}$ requires storing the entire training
set, storage requirements for holding a representation of $\f^{\star}$
is $O(nd$)), which motivates looking for some approximation schemes.
In this section we show how to perform approximate KRR by sampling
the quasimatrix $\matB$ defined in the previous subsection. The resulting
method is actually identical to approximating KRR using random Fourier
features, one of the most popular approximation of KRR, though the
presentation as a sampling method for finding an approximate solution
to a SILR problem is new.

Consider the wide quasimatrix $\matB$ defined in the previous subsection.
A coordinate representation of $\matB$ is
\[
\z(\veta)=\left[\begin{array}{c}
\overline{\varphi(\x_{1},\veta)}\\
\vdots\\
\overline{\varphi(\x_{n},\veta)}
\end{array}\right]\,.
\]
This allows us to discuss column sampling of $\matB$. For $s\leq n$,
consider the matrix $\matB_{\veta}\in\C^{n\times s}$ obtained by
column sampling $\matB$ according to $\mu$. That is, we sample $\veta_{1},\dots,\veta_{s}$
according to $\mu$ and define the matrix 
\[
\matB_{\veta}=\left[\begin{array}{cccc}
\z(\veta_{1}) & \z(\veta_{2}) & \cdots & \z(\veta_{s})\end{array}\right]=\left[\begin{array}{ccc}
\overline{\varphi(\x_{1},\veta_{1})} & \cdots & \overline{\varphi(\x_{1},\veta_{s})}\\
\vdots &  & \vdots\\
\overline{\varphi(\x_{n},\veta_{1})} & \cdots & \overline{\varphi(\x_{n},\veta_{s})}
\end{array}\right]\,.
\]
Let 
\[
\w_{\veta}^{\star}=\arg\min_{\w\in\mathbb{\C}^{s}}\TNormS{\matB_{\veta}\w-\y}+\lambda\XNormS{\w}2\,.
\]
Finding $\w_{\veta}^{\star}$ amounts to solving a finite linear least
squares problem, and can be accomplished using $O(ns^{2})$ arithmetic
operations (and, notably, without performing any FUNOPs). The approximate
KRR estimator is 
\[
\f_{\veta}(\x)=\sum_{i=1}^{s}\overline{\varphi(\x,\eta_{i})}(w_{\veta}^{\star})_{i}
\]
where $(w_{\veta}^{\star})_{i}$ denotes entry $i$ of $\w_{\veta}^{\star}$.
In a sense, the vector $\w_{\veta}^{\star}$ is an approximation of
the function $\w^{\star}$ that is obtained by solving a sampled version
of Eq.~(\ref{eq:l2_krr}), and $\f_{\veta}$ approximates the inner
product $\dotprod{\varphi(\x,\cdot)}{\w^{\star}}{L_{2}({\cal X},d\mu)}$.

\subsection{\label{subsec:stretching_ls}Stretching a Finite Linear Least Squares
Problem}

Since numerical computing is typically done with numbers and not with
functions, it is natural to find an approximate solution to SILR problems
by sampling the quasimatrix. Here, we show that it is also possible
to go the other way, and ``stretch'' a finite linear least squares
problem to a SILR problem. This process is interesting since it yields
a novel interpretation to the use of the Johnson-Lindenstrauss sketch
in order to approximately solve a linear regression problem.

Suppose that $\matX\in\R^{n\times d}$ is a full rank matrix with
$n\gg d$, and that $\y\in\R^{n}$. Consider finding $\w^{\star}$
that minimizes $\TNormS{\matX\w-\y}$. Define the function
\[
\varphi(\x,\veta)=\veta^{\T}\x,
\]
and let $p$ denote the standard Gaussian density over $\R^{n}$.
We have 
\[
\int_{\R^{n}}\varphi(\x,\veta)^{2}p(\veta)d\veta=\x^{\T}\left(\int_{\R^{n}}\veta\veta^{\T}p(\veta)d\veta\right)\x=\x^{\T}\x,
\]
so for $\x\in\R^{n}$ it holds that $\varphi(\x,\cdot)\in L_{2}(\R^{n},d\mu)$
where $d\mu$ denotes the standard Gaussian distribution. Denote by
$\x_{1},\dots,\x_{d}$ the columns of $\matX$. Define the $\infty\times d$
quasimatrix $\matA$ over $L_{2}(\R^{n},d\mu)$: $\matA=\left[\begin{array}{ccc}
\varphi(\x_{1},\cdot) & \cdots & \varphi(\x_{d},\cdot)\end{array}\right]$.We now show that 
\begin{equation}
\w^{\star}=\arg\min_{\w\in\R^{d}}\XNormS{\matA\w-\varphi(\y,\cdot)}{L_{2}(\R^{n},d\mu)}\,.\label{eq:silr-from-finite}
\end{equation}
Indeed, for every $\w\in\R^{d}$ we have 
\begin{align*}
\XNormS{\matA\w-\varphi(\y,\cdot)}{L_{2}({\cal X},d\mu)} & =\int_{\R^{n}}\left(\sum_{i=1}^{d}(\veta^{\T}\x_{i})w_{i}-\veta^{\T}\y\right)^{2}p(\veta)d\veta\\
 & =\int_{\R^{n}}\left(\veta^{\T}\left(\matX\w-\y\right)\right)^{2}p(\veta)d\veta\\
 & =\TNormS{\matX\w-\y}.
\end{align*}
Thus, we have converted the finite linear least squares problem to
a SILR problem.

Let us now consider approximately solving Eq.~(\ref{eq:silr-from-finite})
by sampling ``rows'' from $\matA$ and the corresponding entries
from $\varphi(\y,\cdot)$. Since $\matA$ is a quasimatrix over a
$L_{2}$ space, we need a coordinate representation to meaningfully
talk about sampling rows from $\matA$. A coordinate representation
of $\matA$ is $\z(\veta)=\matX^{\T}\veta$where the index set is
$\Omega=\R$. We can now sample Eq.~(\ref{eq:silr-from-finite})
as follows. We sample $\veta_{1},\dots,\veta_{s}\in\R^{n}$ independently
according to $p$ and form the matrix 
\[
\matA_{\veta}=\left[\begin{array}{c}
\z(\veta_{1})^{\T}\\
\z(\veta_{2})^{\T}\\
\vdots\\
\z(\veta_{s})^{\T}
\end{array}\right]
\]
and the vector $\y_{\veta}=(\veta_{i}^{\T}\y)_{i=1}^{s}$. The sampled
problem (which is again a finite linear least squares problem) is
\[
\hat{\w}=\TNormS{\matA_{\veta}\w-\y_{\veta}}\,.
\]

One may ask, whether $\hat{\w}$ is close to a be a minimizer of $\TNormS{\matX\w-\y}$.
Let $\matS$ be the $s\times n$ matrix whose rows are $\veta_{1},\dots,\veta_{s}$,
then $\matA_{\veta}=\matS\matX$ and $\y_{\veta}=\matS\y$. Using
known results on subspace embedding \cite{Woodruff14} we conclude
that if $s=\Omega(d/\epsilon^{2})$ then with high probability 
\[
\TNorm{\matX\hat{\w}-\y}\leq(1+\epsilon)\TNorm{\matX\w^{\star}-\y}\,.
\]
The random matrix $\matS$ is a Johnson-Lindenstrauss sketching matrix,
and we have demonstrated that applying the Johnson-Lindenstrauss sketch
corresponds to stretching the linear least squares problem and then
applying plain row sampling. A surprising aspect of this observation
is the fact that only $O(d/\epsilon^{2})$ samples are sufficient.
Indeed, standard techniques used for analyzing sampled linear least
squares problems, which are based on matrix tail inequalities, can
be used to derive results that require $\Omega(d\log d/\epsilon^{2})$
samples at the very least; see Section~\ref{sec:sampling}.

\section{\label{sec:Direct}Direct Methods}

Direct methods attempt to compute the solutions of SILR problems using
quasimatrix operations. This mostly involves FUNOPs (under the chebfun
model), but in some cases the computation can be reduced to algorithms
that operate in the standard model (without FUNOPs). For simplicity,
we assume that the quasimatrix involved is either full rank or $\lambda>0$.

\subsection{Overdetermined SILR}

Let $\matA$ be a tall quasimatrix with $n$ columns. We can solve
the SILR problem in Eq.~(\ref{eq:over-silr}) using the normal equations.
The development is essentially the same as for the finite linear least
squares case. Let $f(\x)=\HNormS{\matA\x-\b}+\lambda\TNormS{\x}$
be the objective function. We have, 
\begin{align*}
f(\x) & =\HNormS{\matA\x-\b}+\lambda\TNormS{\x}=\x^{\T}(\matA^{\conj}\matA+\lambda\matI_{n})\x-2\realpart{\x^{\T}\matA^{\conj}\b}+\HNormS{\b}\,.
\end{align*}
Thus the optimum value is obtained as the solution of the following
linear system:
\[
(\matA^{\conj}\matA+\lambda\matI_{n})\x=\matA^{\conj}\b
\]
(easily verified by computing the gradient of $f(\x)$ and equating
to zero). Note that since we assumed that either $\matA$ is full
rank or $\lambda>0$, $\matA^{\conj}\matA+\lambda\matI_{n}$ is invertible.

Thus, we can find the optimal $\x^{\star}$ in the chebfun model by
first computing $\matA^{\conj}\matA+\lambda\matI_{n}$ and $\matA^{\conj}\b$
($n(n+3)/2$ FUNOPs), and then solving an $n\times n$ linear system
($O(n^{3})$ FLOPs). However, using the Gram matrix $\matA^{\conj}\matA$
entails a squaring of the condition number, so the use of a factorization
is preferred numerically. It is simple algebra to show that if $\matA=\matQ\matR$
is a reduced QR factorization of $\matA$ then $\x^{\star}=(\matR+\lambda\matR^{-\conj})^{-1}\matQ^{\conj}\b$,
so $n$ FUNOPs and $O(n^{3})$ FLOPs ($O(n^{2})$ if $\lambda=0$)
are needed once we have a QR factorization. However, this entails
explicitly inverting $\matR^{\conj}$, and that matrix might be ill-conditioned.

If $\matA$ was a matrix, explicitly inverting $\matR$ can be avoiding
by factorizing the augmented matrix 
\[
\hat{\matA}=\left[\begin{smallmatrix}\begin{array}{c}
\matA\\
\sqrt{\lambda}\matI_{n}
\end{array}\end{smallmatrix}\right]\,.
\]
 However, if $\matA$ is a quasimatrix over $\H$ then $\hat{\matA}$
is a quasimatrix over $\H\times\C^{n}$ (if column $j$ of $\matA$
is $\a_{j}$, then column $j$ of the augmented quasimatrix is the
tuple $(\a_{j},\sqrt{\lambda}\e_{j})$ where $\e_{j}$ is the $j$-th
identity vector in $\C^{n}$), possibly making computations more cumbersome\footnote{We remark that the chebfun library does support hybrids of quasimatrices
and a matrices (and calls such objects by the name ``chebmatrix'').}. We now show how it is possible to find $\x^{\star}$ using a QR
factorization of $\matA$ without explicitly inverting $\matR$. This
involves fairly standard linear algebra tricks. Suppose we have a
QR factorization $\matA=\matQ_{\matA}\matR_{\matA}$. Let $\matC$
be the matrix obtained by augmenting $\matR_{\matA}$ with the matrix
$\sqrt{\lambda}\matI_{n}$, and form a QR factorization of it. That
is, 
\[
\matC=\left[\begin{array}{c}
\matR_{\matA}\\
\sqrt{\lambda}\matI_{n}
\end{array}\right]=\matQ_{\matC}\matR_{\matC}\,.
\]
Let $\matQ_{\matC,1}$ denote the top $n$ rows of $\matQ_{\matC}$,
and $\matQ_{\matC,2}$ the bottom. Now,
\[
\left[\begin{array}{c}
\matA\\
\sqrt{\lambda}\matI_{n}
\end{array}\right]=\left[\begin{array}{cc}
\matQ_{\matA}\\
 & \matI_{n}
\end{array}\right]\matC=\left[\begin{array}{cc}
\matQ_{\matA}\\
 & \matI_{n}
\end{array}\right]\left[\begin{array}{c}
\matQ_{\matC,1}\\
\matQ_{\matC,2}
\end{array}\right]\matR_{\matC}=\left[\begin{array}{c}
\matQ_{\matA}\matQ_{\matC,1}\\
\matQ_{\matC,2}
\end{array}\right]\matR_{\matC}\,.
\]
Also, since
\[
\left[\begin{array}{cc}
\matQ_{\matC,1}^{\conj}\matQ_{\matA}^{\conj} & \matQ_{\matC,2}^{\conj}\end{array}\right]\left[\begin{smallmatrix}\begin{array}{c}
\matQ_{\matA}\matQ_{\matC,1}\\
\matQ_{\matC,2}
\end{array}\end{smallmatrix}\right]=\matI_{n}\,,
\]
we have a QR factorization of $\hat{\matA}$. This implies that $\x^{\star}$
is the solution of the triangular system
\[
\matR_{\matC}\x=\matQ_{\matC,1}^{*}\matQ_{\matA}^{*}\b\,.
\]

As is the case of finite linear least squares, a reduced SVD can be
used to solve SILR problems as well. If $\matA=\matU\Sigma\matV^{\conj}$
is a reduced SVD factorization, then simple algebra reveals that $\x^{\star}=\matV(\Sigma^{2}+\lambda\matI_{n})^{-1}\Sigma\matU^{\conj}\b$,
so $n$ FUNOPs and $O(n^{2})$ FLOPs are needed once we have an SVD
factorization.

In certain cases it might be possible to compute $\matA^{\conj}\matA$
and $\matA^{\conj}\b$ analytically, without resorting to FUNOPs.
We give a concrete example later, when we discuss the underdetermined
case.

\subsection{\label{subsec:direct-under}Underdetermined SILR}

Let $\matA$ be a wide quasimatrix with $n$ rows. Again, the following
argument follows closely the one used for finite linear least squares.
The space $\range{\matA^{\conj}}$ is a closed linear subspace of
${\cal H}$, so we can write $\x^{\star}=\matA^{\conj}\y^{\star}+\z^{\star}$
where $\y\in\R^{n}$ and $\z\perp\range{\matA^{\conj}}.$ Since $\matA$,
viewed as an operator, is bounded, $\nully{\matA}=\left(\range{\matA^{\conj}}\right)^{\perp}$,
so $\matA\z=0$. Thus, the objective at $\x^{\star}$ is 
\[
\TNormS{\matA\x^{\star}-\b}+\lambda\HNormS{\x^{\star}}=\TNormS{\matA\matA^{\conj}\y^{\star}-\b}+\lambda\HNormS{\matA^{\conj}\y^{\star}}+\lambda\HNormS{\z^{\star}},
\]
where we used the fact that $\z^{\star}\perp\matA^{\conj}\y^{\star}.$
Obviously, $\z^{\star}=0$, otherwise the objective can be reduced.
Denoting $\matK=\matA\matA^{\conj}\in\R^{n\times n}$, we find that
$\y^{\star}$ is the minimizer of 
\[
f(\y)=\TNormS{\matK\y-\b}+\lambda\y^{\T}\matK\y\,.
\]
This can be written as a determined (for $\lambda=0)$ or overdetermined
($\lambda>0)$ finite linear least squares problems, from which we
find that $\y^{\star}$ solves the equation 
\[
(\matK^{2}+\lambda\matK)\y=\matK\b.
\]
Since we assumed that either $\matA$ is full rank or $\lambda>0$,
$\matK+\lambda\matI_{n}$ is invertible, and the vector $(\matK+\lambda\matI_{n})^{-1}\b$
solves the equation. The solution is unique, so $\y^{\star}=(\matK+\lambda\matI_{n})^{-1}\b$.
We find that 
\[
\x^{\star}=\matA^{\conj}(\matK+\lambda\matI_{n})^{-1}\b\,.
\]
Thus, in the chebfun model we can find the optimal $\x^{\star}$ by
first computing $\matK+\lambda\matI_{n}$ ($n(n+1)/2$ FUNOPs), solving
for $\y^{\star}$ ($O(n^{3})$ FLOPs), and finally computing $\matA^{\conj}\y^{\star}$
($n$ FUNOPs).

We can avoid forming the potentially ill-conditioned matrix $\matK$
using a QR factorization of $\matA^{\conj}$ in a way similar to the
previous subsection, where now we have a QR factorization $\matA^{*}=\matQ_{\matA}\matR_{\matA}$.
Similarly, we can factorize
\[
\left[\begin{array}{c}
\matA^{*}\\
\sqrt{\lambda}\matI_{n}
\end{array}\right]=\left[\begin{array}{c}
\matQ_{\matA}\matQ_{\matC,1}\\
\matQ_{\matC,2}
\end{array}\right]\matR_{\matC}\,.
\]
This implies that 
\[
\left[\begin{array}{cc}
\matA & \sqrt{\lambda}\matI_{n}\end{array}\right]=\matR_{\matC}^{\conj}\left[\begin{array}{cc}
\matQ_{\matC,1}^{\conj}\matQ_{\matA}^{\conj} & \matQ_{\matC,2}^{\conj}\end{array}\right]
\]
is an LQ factorization of the augmented matrix. Hence, $\x^{\star}$
is equal to 
\[
\x^{\star}=\matQ_{\matA}\matQ_{\matC,1}(\matR_{\matC}^{\conj})^{-1}\b\,.
\]
An SVD factorization can be used as well: if $\matA^{\conj}=\matU\Sigma\matV^{\conj}$
is an SVD factorization, then $\x^{\star}=\matU(\Sigma^{2}+\lambda\matI_{n})^{-1}\Sigma\matV^{\conj}\b$.

In certain cases, it is possible to compute $\matK$ analytically.
As an example, consider again kernel ridge regression (Section~\ref{subsec:krr})
in the RKHS formulation (Eq.~(\ref{eq:krr-rkhs})). Due to the definition
of $\H_{k}$, the $ij$ entry of $\matK$ is 
\[
\matK_{ij}=\dotprod{k(\x_{i},\cdot)}{k(\x_{j},\cdot)}{\H_{k}}=k(\x_{i},\x_{j})\,,
\]
thus we can form $\matK+\lambda\matI_{n}$ and compute $\alpha=(\matK+\lambda\matI_{n})^{-1}\y$
using $O(n^{2}d+n^{3})$ FLOPs. The solution to Eq.~(\ref{eq:krr-rkhs})
is then 
\[
\f^{\star}(\x)=\sum_{i=1}^{n}\alpha_{i}k(\x_{i},\x)\,.
\]
and all computations are done in the standard model.

\section{Iterative Methods}

In this section we discuss solving SILR problems using iterative methods.
First, we consider using the classical approach of Krylov subspace
methods. We show how methods such as LSMR or LSQR can be rather naturally
generalized to quasimatrices. Next, we propose a novel method based
on stochastic optimization which requires considerably less function
operations, but depends on the ability to sample the quasimatrix.
We also show that this algorithm can, in certain cases, be applied
in the standard model (without FUNOPs).

\subsection{\label{subsec:Krylov_LSMR}Krylov Subspace Methods}

Krylov subspace methods are one of the most important classes of iterative
methods in numerical linear algebra. Many of the most widely used
iterative linear solvers are Krylov subspace methods. One important
benefit of Krylov subspace methods is that they only use matrix-vector
operations. For SILR, when working in the chebfun model, this implies
that each iteration does only $O(n)$ FUNOPs.

It was already observed by several authors that it is possible to
generalize Krylov subspace methods to operator equations. For example,
Olver suggested the use of GMRES with the differentiation operator
\cite{olver2009}, and the chebfun library implements GMRES for operator
equations. Continuous analogues of CG, GMRES and MINRES appear in
\cite{gilles2019continuous} in the context of differential operators.
In the same vein, we can adapt Krylov subspace algorithms for finite
linear least squares, such as LSQR~\cite{paige1982lsqr} and LSMR~\cite{fong2011lsmr},
to solve SILR problems. Here we describe the LSMR algorithm for quasimatrices.
The development is a rather straightforward generalization of the
matrix case, but we show it for concreteness.

\subsubsection{Golub-Kahan Bidiagonalization Process for Quasimatrices}

LSMR and LSQR are based on Golub-Kahan bidiagonalization \cite{golub1965calculating}.
The goal of Golub-Kahan bidiagonalization is to iteratively find a
decomposition $\matU^{*}\matA\matV=\matB$ where $\matU$ and $\matV$
have orthogonal columns, and $\matB$ is a bidiagonal matrix. When
$\matA$ is a quasimatrix over ${\cal H}$, one of $\matU$ and $\matV$
is a quasimatrix over ${\cal H}$ and the other one is a matrix. The
algorithm remains essentially unchanged, and is given in Algorithm~\ref{alg:bidiag1}
for a tall quasimatrix (for a wide quasimatrix the algorithm remains
the same under the corresponding changes of norms).

After $k$ steps of the algorithm, we have $\matA\matV_{k}=\matU_{k+1}\matB_{k}$
and $\matA^{*}\matU_{k+1}=\matV_{k+1}\matL_{k+1}^{*}$. Note that,
in the overdetermined case, where $\matA$ is a $\infty\times n$
quasimatrix, $\matV_{k}$ is a $n\times k$ matrix and $\matU_{k}$
is a $\infty\times k$ tall quasimatrix over $\H$. In the underdetermined
case, $\matU_{k}$ is a $n\times k$ matrix and $\matV_{k}$ is a
$\infty\times k$ tall quasimatrix over ${\cal H}$ such that $\matV_{k}^{*}\matV_{k}=\matU_{k}^{*}\matU_{k}=\matI_{k}$.
The algorithm also defines a $(k+1)\times k$ lower bidiagonal matrix
$\matB_{k}$.

\begin{algorithm}[t]
\begin{algorithmic}[1]

\STATE \textbf{Inputs:}

Tall $\infty\times n$ quasimatrix $\matA$ over ${\cal H}$, and
$\infty\times1$ quasimatrix $\b$.

\STATE \textbf{set:
\[
\beta_{1}=\HNorm{\b},\quad\u_{1}=\b/\beta_{1},\quad\alpha_{1}=\TNorm{\matA^{\T}\u_{1}},\quad\v_{1}=\matA^{\T}\u_{1}/\alpha_{1}\,.
\]
}

\STATE \textbf{for} $k=1,2,\ldots,$

\textbf{
\begin{align*}
\beta_{k+1}=\ONorm{\matA\v_{k}-\alpha_{k}\u_{k}}{\H}, & \u_{k+1}=(\matA v_{k}-\alpha_{k}\u_{k})/\beta_{k+1}\\
\alpha_{k+1}=\TNorm{\matA^{\conj}\u_{k+1}-\beta_{k+1}\v_{k}}, & \v_{k+1}=(\matA^{\conj}\u_{k+1}-\beta_{k+1}\v_{k})/\alpha_{k+1}\,.
\end{align*}
}

\STATE $\qquad$\textbf{Denote:
\begin{align*}
\matV_{k}=\left[\begin{array}{cccc}
\v_{1} & \mat v_{2} & \dots & \mat v_{k}\end{array}\right], & \matU_{k}=\left[\begin{array}{cccc}
\u_{1} & \mat u_{2} & \dots & \mat u_{k}\end{array}\right]\\
\matB_{k}=\left[\begin{array}{cccc}
\alpha_{1}\\
\beta_{2} & \alpha_{2}\\
 & \ddots & \ddots\\
 &  & \beta_{k} & \alpha_{k}\\
 &  &  & \beta_{k+1}
\end{array}\right], & \matL_{K}=\left[\begin{array}{cc}
\matB_{k} & \alpha_{k+1}\e_{k+1}\end{array}\right]\,.
\end{align*}
}

\end{algorithmic}

\caption{\label{alg:bidiag} (Tall) Quasimatrix Golub-Kahan Bidiagonalizaion.}
\end{algorithm}

\subsubsection{LSMR for Overdetermined SILR}

Recall that the solution $\x^{\star}$ of an overdetermined SILR problem
solves the normal equations $(\matA^{\conj}\matA+\lambda\matI_{n})\x^{\star}=\matA^{\conj}\b$.
LSMR is equivalent to applying MINRES to the normal equations, i.e.
in each iteration the minimizer of $\matA^{\conj}\b-(\matA^{\conj}\matA+\lambda\matI_{n})\x$
is found under the constraint that $\x$ belongs to the Krylov subspace.
Thus, defining $\rb_{k}=\b-\matA\x_{k}$ at iteration $k$, LSMR minimizes
$\TNorm{\matA^{\conj}\rb_{k}-\lambda\x_{k}}$ subject to $\x_{k}\in{\cal K}_{k}(\matA^{\conj}\matA,\b)$,
where ${\cal K}_{k}(\matA^{\conj}\matA,\b)\coloneqq\text{span}\left\{ \b,\matA^{\conj}\matA\b,\dots,(\matA^{\conj}\matA)^{k-1}\b\right\} $
is the order $k$-th order Krylov subspace generated by $\matA^{\conj}\matA$
and $\b$.

To find $x_{k}$, LSMR uses the Golub-Kahan bidiagonalization. After
$k$ iterations we have $\matA\matV_{k}=\matU_{k+1}\matB_{k}$ and
$\matA^{*}\matU_{k+1}=\matV_{k+1}\matL_{k+1}^{*}$, where $\matV_{k}$
is a $n\times k$ matrix and $\matU_{k}$ is a $\infty\times k$ tall
quasimatrix over $\H$. Thus,
\begin{align*}
\matA^{*}(\mat A\matV_{k}) & =(\matA^{*}\matU_{k+1})\matB_{k}=\matV_{k+1}\matL_{k+1}^{\T}\matB_{k}=\matV_{k+1}\left[\begin{array}{c}
\matB_{k}^{\T}\matB_{k}\\
\alpha_{k+1}\beta_{k+1}\e_{k+1}^{\T}
\end{array}\right]\\
\matA^{*}\b & =\beta_{1}\matA^{\T}\u_{1}=\alpha_{1}\beta_{1}\matV_{k+1}\e_{1}\,.
\end{align*}
(These equations are the same as in the matrix case: the quasimatrix
algebra defined in Section~\ref{sec:quasimatrices} allows us to
write essentially the same derivations). Since $\x_{k}$ is in the
Krylov subspace, we can write $\x_{k}=\matV_{k}\y_{k}$ for some $\y_{k}\in\mathbb{\C}^{k}$.
Thus, we can write
\begin{align*}
\XNormS{\matA^{*}(\matA\x_{k}-\b)}2+\lambda\XNormS{\x_{k}}2 & =\XNormS{\matA^{*}(\matA\matV_{k}\y_{k}-\b)}2+\lambda\XNormS{\matV_{k}\y_{k}}2\\
 & =\left\Vert \matV_{k+1}\left(\left[\begin{array}{c}
\matB_{k}^{\T}\matB_{k}\\
\alpha_{k+1}\e_{k+1}^{\T}
\end{array}\right]\y_{k}-\alpha_{1}\beta_{1}\e_{1}\right)\right\Vert _{2}^{2}+\lambda\TNormS{\y_{k}}\,.
\end{align*}
So, finding $\y_{k}$ and $\x_{k}$ has been reduced to the solution
of a finite linear least squares problem. An algorithm for finding
these vectors efficiently and iteratively is described in \cite{fong2011lsmr}.

\textbf{Stopping criteria: }the Golub-Kahan process terminates whenever
$\alpha_{k+1}=0$ or $\beta_{k+1}=0$, which implies that the last
equation is zero. However, we can use one of the stopping criteria
originally presented for the LSQR algorithm, involving the predetermined
parameters $\text{ATOL},\text{BTOL}$ and $\text{CONLIM}$:
\begin{align*}
S_{1}: & \text{Stop if}\quad\sqrt{\XNormS{\rb_{k}}{\H}+\lambda\TNormS{\x_{k}}}\leq\text{BTOL}\ONorm{\b}{\H}+\text{ATOL}\sqrt{\sigma_{\max}(\matB_{k})^{2}+\lambda}\ONorm{\x_{k}}2\\
S_{2}: & \text{Stop if}\quad\ONorm{\matA^{\conj}\rb_{k}-\lambda\x_{k}}2\leq\text{ATOL}\sqrt{\sigma_{\max}(\matB_{k})^{2}+\lambda}\sqrt{\XNormS{\matA\x_{k}-\b}{\H}+\lambda}\\
S_{3}: & \text{Stop if}\quad\sqrt{\frac{\sigma_{\max}(\matB_{k})^{2}+\lambda}{\sigma_{\min}(\matB_{k})^{2}+\lambda}}\geq\text{CONLIM}\,.
\end{align*}
The motivation for these stopping rules is the fact that $\sigma_{\max}(\matB_{k})$
and $\sigma_{\min}(\matB_{k})$ provide estimates for $\sigma_{\max}(\matA)$
and $\sigma_{\min}(\matA)$. This follows from the fact that $\matB_{k}^{\T}\matB_{k}=\matV_{k}^{*}\matA^{*}\matA\matV_{k}$.
See \cite{fong2011lsmr} for more details.

\textbf{Complexity:} when compared to the matrix version of LSMR,
the quasimatrix version trades each matrix-vector product with $n$
FUNOPs. Thus, in terms of FUNOPs, $2n$ FUNOPs are required per iteration.
Since the number of iterations is $O(\sqrt{\kappa(\matA^{\conj}\matA+\lambda\matI_{n}}),$
overall complexity is $O(n\sqrt{\kappa(\matA^{\conj}\matA+\lambda\matI_{n}})$
FUNOPs.

\subsubsection{Numerical Example}

We illustrate the use of LSMR for the problem of approximating the
Runge function $f(x)=1/(1+25x^{2})$ on $[-1,1]$ using a polynomial
of degree $300$. We can write the approximation as the solution of
an overdetermined SILR problem where $\matA$ is any quasimatrix whose
columns span the space of degree 300 polynomials. However, we want
$\matA$ to be reasonably well conditioned so that LSMR will converge
quickly, so we use Chebyshev polynomials as the columns of $\matA$
(we empirically observed that when the columns are the Chebyshev polynomials,
$\matA$ is well-conditioned, though we are unaware of any analytical
result showing this; note that taking the normalized Legendre polynomials
instead would have resulted in an orthogonal $\matA$, which would
have made for an uninteresting numerical example). Thus, we solve
the SILR problem where 
\[
\matA=\left[\begin{array}{cccc}
T_{0} & T_{1} & \dots & T_{299}\end{array}\right],\quad\b=\left[\frac{1}{1+25x^{2}}\right]\,.
\]
In the above, $T_{j}$ is the $j$-th Chebyshev polynomial. We use
$\lambda=0$ (no regularization) and parameters $\text{ATOL}=\text{BTOL}=10^{-7}$.
Convergence plots are shown in Figure~\ref{fig:LSMR_errors}.

\begin{figure}[H]
\begin{centering}
\begin{tabular}{c}
\includegraphics[width=0.45\textwidth]{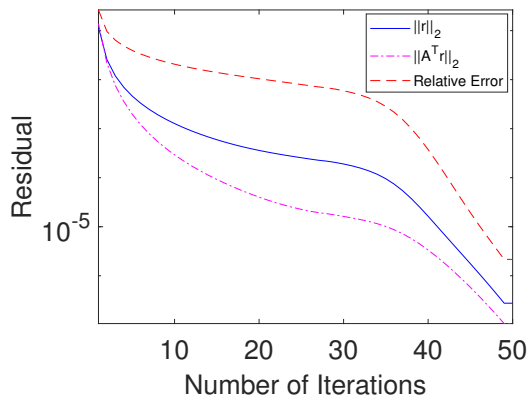}\tabularnewline
\end{tabular}
\par\end{centering}
\caption{\label{fig:LSMR_errors} Numerical illustration: using LSMR to solve
a SILR related to approximating the Runge function using a polynomial.}
\end{figure}

\subsection{\label{subsec:SVRG_SILR}Stochastic Variance Reduced Gradient (SVRG)}

Recent literature on convex optimization advocated the use of stochastic
methods. Even for the specialized cases of solving linear equations
or linear least squares, such methods have been shown to be beneficial
\cite{LS13,gonen16svrg}. In this section, we propose a method for
solving SILR problems using Stochastic Variance Reduced Gradient (SVRG)~\cite{johnson2013accelerating,xiao2014proximal}.
SVRG is a stochastic optimization method for minimizing objective
functions that have finite sum structure, i.e. of the form
\[
f(\x)=\frac{1}{n}\sum_{i=1}^{n}f_{i}(\x)\,.
\]
For such objective functions, we can compute stochastic gradients
by sampling an index of the sum. SVRG's main benefit comes from the
fact that it combines such stochastic gradients with a small amount
of full gradients (i.e. exact gradients of $f$) . For strongly convex
functions, the number of such full gradients we need to compute is
independent of the condition number (however, the number of stochastic
gradients does depend on the average condition number).

For SILR, full gradients correspond to products of a quasimatrix $\matA$
with a vector or function, but this is the only operation that assumes
the chebfun model and requires FUNOPs. Thus, by using SVRG we remove
the condition number dependence for the number of FUNOPs required
for convergence, which is a major improvement over Krylov methods.
However, stochastic gradients in SVRG correspond to sampling objective
functions, and for SILR this translates to sampling a row from a tall
quasimatrix or a column from a wide quasimatrix. Thus, the quasimatrix
must be a quasimatrix over a $L_{2}$ space, and must have a coordinate
representation.

One obstacle in applying SVRG to SILR problems is that such problems
cannot be written as a finite sum, but rather can be written as an
integral of simpler functions, i.e. 
\[
f(\x)=\int_{\Omega}f_{\veta}(\x)d\mu(\veta)\,.
\]
We generalize SVRG and its analysis to handle such functions. The
generalization might be of independent interest, and appears in Appendix~\ref{sec:SVRG_int}.

\subsubsection{Overdetermined SILR}

Consider the overdetermined SILR problem (Eq.~(\ref{eq:over-silr}))
where the quasimatrix $\matA$ is over ${\cal H}=L_{2}(\Omega,d\mu)$
for some index set $\Omega$. We further assume we have a coordinate
representation $\z_{\matA}:\Omega\to\mathbb{\mathbb{R}}^{n}$ for
$\matA$ and $\z_{\b}:\Omega\to\mathbb{\mathbb{R}}$ for $\b$. We
further assume there exists a $M$ such that for every $\veta\in\Omega$
we have $\TNormS{\z_{\matA}(\veta)}\leq M$. We can write the objective
function in Eq. (\ref{eq:over-silr}) as an integral:
\begin{align*}
\frac{1}{2}\XNormS{\matA\x-\b}{L_{2}(\Omega,d\mu)}+\frac{\lambda}{2}\TNormS{\x} & =\frac{1}{2}\XNormS{\sum_{i=1}^{n}x_{i}\overline{\z_{\matA}(\cdot)_{i}}-\b}{L_{2}(\Omega,d\mu)}+\frac{\lambda}{2}\TNormS{\x}\\
 & =\frac{1}{2}\XNormS{\z_{\matA}(\cdot)^{*}\x-\b}{L_{2}(\Omega,d\mu)}+\frac{\lambda}{2}\TNormS{\x}\\
 & =\frac{1}{2}\int_{\Omega}\left(\z_{\matA}(\veta)^{*}\x-\z_{\b}(\veta)\right)^{2}+\lambda\TNormS{\x}\,d\mu(\veta)\\
 & =\int_{\Omega}f_{\veta}(\x)d\mu(\veta)
\end{align*}
where 
\[
f_{\veta}(\x)\coloneqq\frac{1}{2}\left(\z_{\matA}(\veta)^{*}\x-\z_{\b}(\veta)\right)^{2}+\frac{\lambda}{2}\TNormS{\x}\,.
\]

We can now apply the aforementioned variant of SVRG \cite{johnson2013accelerating}
(see Appendix~\ref{sec:SVRG_int}), which is adapted for objective
integrable functions. To do so, the following assumptions need to
be verified:
\begin{assumption}
\label{assu:Lip-eta}For all $\veta\in\Omega$, $\nabla f_{\veta}(\x)$
is Lipschitz continuous, i.e., there exists $L_{\veta}>0$ such that
for all $\x,\y\in\mathbb{R}^{n}$ 
\[
\Vert\nabla f_{\veta}(\x)-\nabla f_{\veta}(\y)\Vert\leq L_{\veta}\Vert\x-\y\Vert\,.
\]
\end{assumption}

\begin{assumption}
\label{assu:strong-convex}Suppose that $f(\x)$ is strongly convex,
i.e., there exist $\gamma>0$ such that for all $\x,\y\in\mathbb{R}^{n}$
\[
f(\x)-f(\y)\geq\frac{\gamma}{2}\Vert\x-\y\Vert_{2}^{2}+\nabla f(\boldsymbol{y})^{\T}(\x-\y)\,.
\]
\end{assumption}

\begin{assumption}
\label{assu:grad-int}The equality $\nabla f(\x)=\int_{\Omega}\nabla f_{\veta}(\x)d\mu(\veta)$
hold.
\end{assumption}

\begin{assumption}
\label{assu:finite-L_eta}$L_{\sup}\coloneqq\sup_{\veta\in\Omega}L_{\veta}<\infty$.
\end{assumption}

We begin by writing
\[
f(\x)\coloneqq\frac{1}{2}\XNormS{\matA\x-\b}{L_{2}(\Omega,d\mu)}+\frac{\lambda}{2}\TNormS{\x}=\frac{1}{2}\x^{\T}(\matK+\lambda\matI_{n})\x-\x^{\T}\matA^{*}\b+\frac{1}{2}\XNormS{\b}{L_{2}(\Omega,d\mu)}
\]
where $\matK=\matA^{\conj}\matA\in\mathbb{R}^{n\times n}$. Thus,
\[
\nabla f(\x)=\matA^{\conj}(\matA\x-\b)+\lambda\x\,.
\]
It can be seen that Assumption~\ref{assu:strong-convex} holds with
$\gamma=\lambda+\lambda_{\min}(\matK)$. We also have 
\[
\nabla f_{\veta}(\x)=\z_{\matA}(\veta)\left(\z_{\matA}(\veta)^{*}\x-\z_{\b}(\veta)\right)+\lambda\x
\]
with 
\begin{align*}
\int_{\Omega}\nabla f_{\veta}(\x)d\mu(\veta) & =\int_{\Omega}\z_{\matA}(\veta)\left(\z_{\matA}(\veta)^{*}\x-\z_{\b}(\eta)\right)+\lambda\x d\mu(\veta)\\
 & =\left(\int_{\Omega}\z_{\matA}(\veta)\z_{\matA}(\veta)^{*}d\mu(\veta)\right)\x-\int_{\Omega}\z_{\matA}(\veta)\z_{\b}(\veta)d\mu(\veta)+\lambda\x\\
 & =\matK\x-\matA^{\conj}\b+\lambda\x=\nabla f(\x)\,.
\end{align*}
so Assumption~\ref{assu:grad-int} holds as well. Note that for every
$\veta\in\X$ 
\begin{align*}
\TNorm{\nabla f_{\veta}(\x)-\nabla f_{\veta}(\y)} & =\TNorm{(\z_{\matA}(\veta)\z_{\matA}(\veta)^{*}+\lambda\matI_{n})(\x-\y)}\\
 & \leq\left(\TNormS{\z_{\matA}(\veta)}+\lambda\right)\TNorm{\x-\y}
\end{align*}
so each $\nabla f_{\veta}$ is Lipschitz continuous with Lipschitz
constant $L_{\veta}=\TNormS{\z_{\matA}(\veta)}+\lambda$. Thus, Assumptions
\ref{assu:Lip-eta} and \ref{assu:finite-L_eta} hold with $L_{\sup}=M+\lambda$.

Therefore, according to Theorem \ref{thm:converge_proof} (in Appendix~\ref{sec:SVRG_int}),
if we set 
\[
m=50\cdot\kappa,\;\kappa=\frac{M+\lambda}{\gamma^{2}+\lambda},\;\;\alpha=\frac{\theta}{M+\lambda},\;0<\theta<\frac{1}{4}
\]
where $\gamma$ is any lower bound on $\sigma_{\min}(\matA)$ (if
$\lambda>0$ we can take $\gamma=0$), then taking $\theta=1/5$ and
assuming we start with $\x=0$ yields
\[
\mathbb{E}\left[f(\tilde{\x}_{s})\right]-f(\x^{\star})\leq\left(\frac{5}{6}\right)^{s}\left(\frac{1}{2}\XNormS{\b}{L_{2}(\Omega,d\mu)}-f(\x^{\star})\right)\,.
\]

Overall, to reduce (in expectation) by a factor of $\epsilon$ we
need to do $O(\log(1/\epsilon))$ outer iterations, each requiring
$2n$ FUNOPs. Each outer iteration requires $O\left(\kappa\right)$
inner iterations, each requiring $O(n+T)$ FLOPS where $T$ is the
cost of computing $\z_{\matA}(\veta)$ and $\z_{\b}(\veta)$ for a
given $\veta$, so in total we need $O\left((n+T)\cdot\kappa\cdot\log(1/\epsilon)\right)$
FLOPs. We see that in contrast with Krylov subspace methods, the number
of FUNOPs does \emph{not} depend on the condition number. The proposed
algorithm is summarized in Algorithm~\ref{alg:svrg_int-tall}.

\begin{algorithm}[t]
\begin{algorithmic}[1]

\STATE \textbf{Inputs:}

- Tall $\infty\times n$ quasimatrix $\matA$ over $L_{2}(\Omega,d\mu)$,
along with coordinate representation $\z_{\matA}:\Omega\to\mathbb{C}^{n}$\textbf{
}\\
- $\b$ with coordinate representation $\z_{\b}:\Omega\to\C^{n}$,
$\lambda>0$\\
- $M$ such that for all $\veta\in\Omega$ we have $\TNormS{\z_{\matA}(\veta)}\leq M$
\\
- $\gamma$ such that $0\leq\gamma\leq\sigma_{\min}(\matA)$\\
- Accuracy parameter $\epsilon>0$

\STATE $\tilde{\x}_{0}\gets0$

\STATE $\alpha\gets\frac{1}{5(M+\lambda)}$, $m\gets\frac{50(M+\lambda)}{\gamma^{2}+\lambda}$

\STATE $s_{\max}\gets\left(\log\left(\frac{6}{5}\right)\right)^{-1}\cdot\log\left(\frac{\XNormS{\b}{L_{2}(d\mu)}}{2\epsilon}\right)$

\STATE \textbf{Iterate: }for $s=1,2,\ldots,s_{\max}$

\STATE $\qquad$$\tilde{\x}=\tilde{\x}_{s-1}$

\STATE $\qquad$$\tilde{\g}=\matA^{\conj}(\matA\tilde{\x}-\b)+\lambda\tilde{\x}$

\STATE $\qquad$$\x_{0}=\tilde{\x}$

\STATE $\qquad$\textbf{Iterate: }for $k=1,2,\ldots,m$

\STATE $\qquad\qquad$sample $\veta_{k}$ from the distribution $\mu$

\STATE $\qquad\qquad$$\x_{k}=\x_{k-1}-\alpha\left(\z_{\matA}(\veta_{k})\z_{\matA}(\veta_{k})^{*}\left(\x_{k-1}-\tilde{\x}\right)+\lambda\left(\x_{k-1}-\tilde{\x}\right)+\tilde{\g}\right)$

\STATE \textbf{$\qquad$end}

\STATE \textbf{$\qquad$option I: }set $\tilde{\x}_{s}=\x_{m}$

\STATE \textbf{$\qquad$option II: }set $\tilde{\x}_{s}=\frac{1}{m}\sum_{k=1}^{m}\x_{k}$

\STATE \textbf{end}

\RETURN $\tilde{\x}_{s_{\max}}$

\end{algorithmic}

\caption{\label{alg:svrg_int-tall}SVRG for overdetermined SILR.}
\end{algorithm}

\subsubsection{Underdetermined SILR}

We now consider the case that $\matA$ is wide $n\times\infty$ quasimatrix
over $L_{2}(\Omega,d\mu)$ of full rank, and $\b\in\C^{n}$. As explained
in subsection~\ref{subsec:direct-under}, the optimal solution $\x^{\star}$
has the form $\x^{\star}=\matA^{\conj}\y^{\star}$ for $\y^{\star}\in\C^{n}$.
In addition, we have $\y^{\star}=(\matK+\lambda\matI_{n})^{-1}\b$
where $\matK=\matA\matA^{\conj}$. Hence, 
\[
\y^{\star}=\arg\min_{\y\in\C^{n}}\frac{1}{2}\y^{\conj}(\matK+\lambda\matI_{n})\y-\y^{\conj}\b\,.
\]
Thus, we can find approximate solutions to the regression problem
by optimizing
\[
f(\y)\coloneqq\frac{1}{2}\y^{\conj}(\matK+\lambda\matI_{n})\y-\y^{\conj}\b
\]
and returning $\tilde{\x}=\matA^{*}\tilde{\y}$ for the $\tilde{\y}$
found by the optimization process. Note that if we find a $\tilde{\y}$
such that $f(\tilde{\y})\leq f(\y^{\star})+\epsilon$ then for $\tilde{\x}=\matA^{\conj}\tilde{\y}$
we have 
\[
\XNormS{\tilde{\x}-\x^{\star}}{L_{2}(\Omega,d\mu)}\leq\frac{2\lambda_{\max}(\matK)}{\lambda_{\min}(\matK)+\lambda}\cdot\epsilon\,.
\]

We can again use SVRG (with the specific variant described in Appendix~\ref{sec:SVRG_int})
to minimize $f(\y)$. Since the assumptions are the same as in the
previous section, and the developments are almost identical, we do
not repeat them. The algorithm is almost identical to Algorithm~\ref{alg:svrg_int-tall},
with  two small differences: the equation for $s_{\max}$ is replaced
by $(\log(6/5))^{-1}\cdot\log(\TNormS{\b}+\TNormS{\b}/(\epsilon\sqrt{\lambda_{\min}(\matK)+\lambda}))$,
and $\tilde{\g}$ is  $\matA(\matA^{\conj}\tilde{\y})+\lambda\tilde{\y}-\b.$

\subsubsection{SVRG for Kernel Ridge Regression}

Recall that KRR can be recasted as an underdetermined SILR problem
(subsection \ref{subsec:krr}). We can use the algorithm from the
previous subsection to solve this SILR problem. However, since we
can compute $\matK$ via the kernel function without assuming the
chebfun model, we can avoid performing FUNOPs when computing $\tilde{\g}.$
That is, we can apply SVRG under the standard model. For this case,
the assumptions hold with $M=d$.

We illustrate the performance of this algorithm on a small scale experiment.
The goal is to learn a one dimensional dataset generated by noisily
sampling the function $f^{\star}(x)=\sin(6x)+\sin(60e^{x})$, i.e.
$y_{i}=f^{\star}(x_{i})+\epsilon_{i}$ with $\epsilon_{i}\overset{i.i.d}{\sim}\mathcal{N}(0,0.3^{2})$.
The training set consists of 400 equispaced examples on $[-1,1]$,
and we use the Gaussian kernel. The goal was to reach error $\epsilon=10^{-2}$.
The experiment was run with fixed step size of $\alpha=10^{-4}<1/2L_{\sup}$
. We varied both the value of $m$ and $s.$ Results are reported
in Figure \ref{fig:SVRG-Test-Error}.

\begin{figure}[t]
\begin{centering}
\begin{tabular}{ccc}
\includegraphics[width=0.4\textwidth]{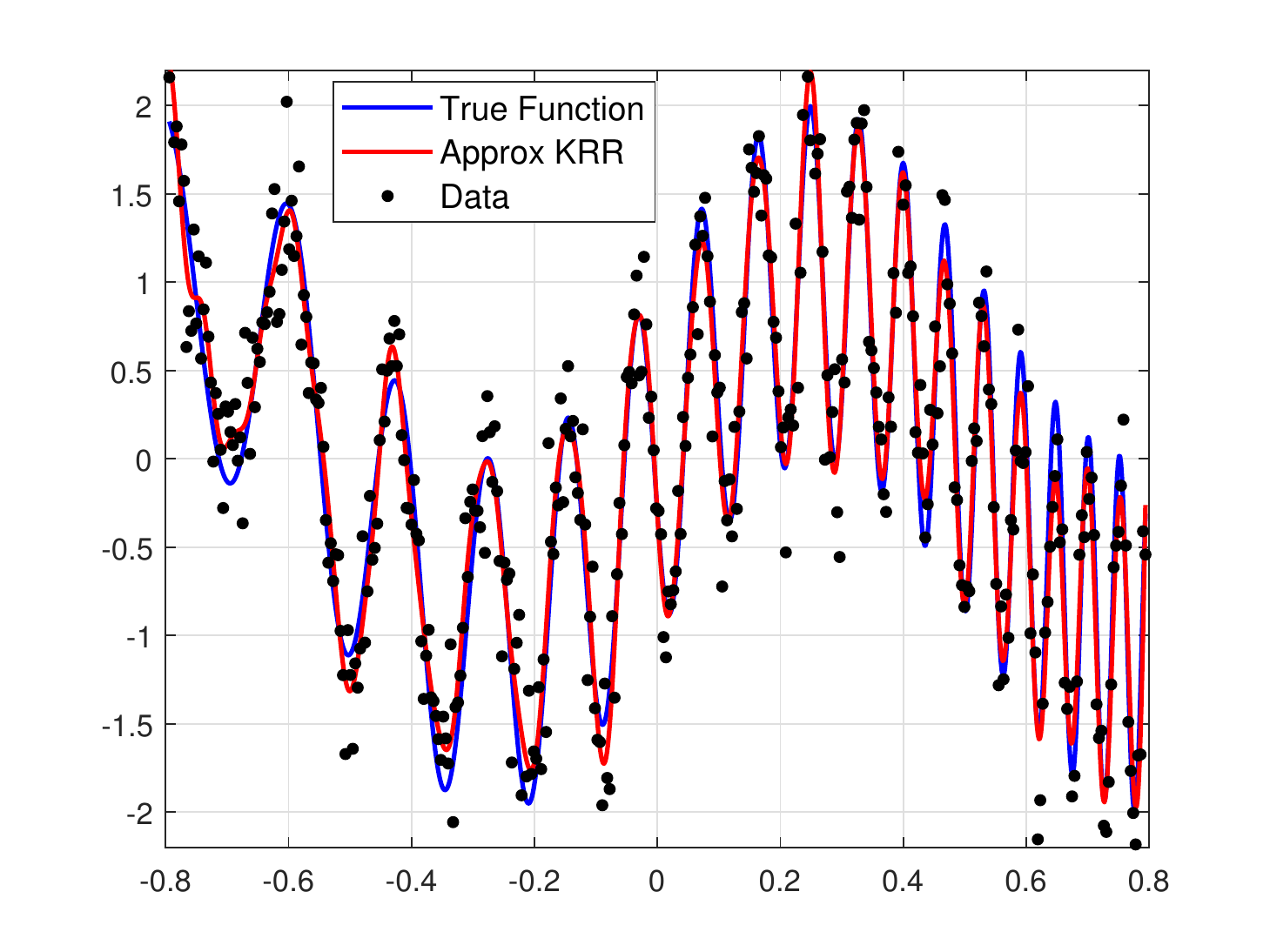} & ~ & \includegraphics[width=0.4\textwidth]{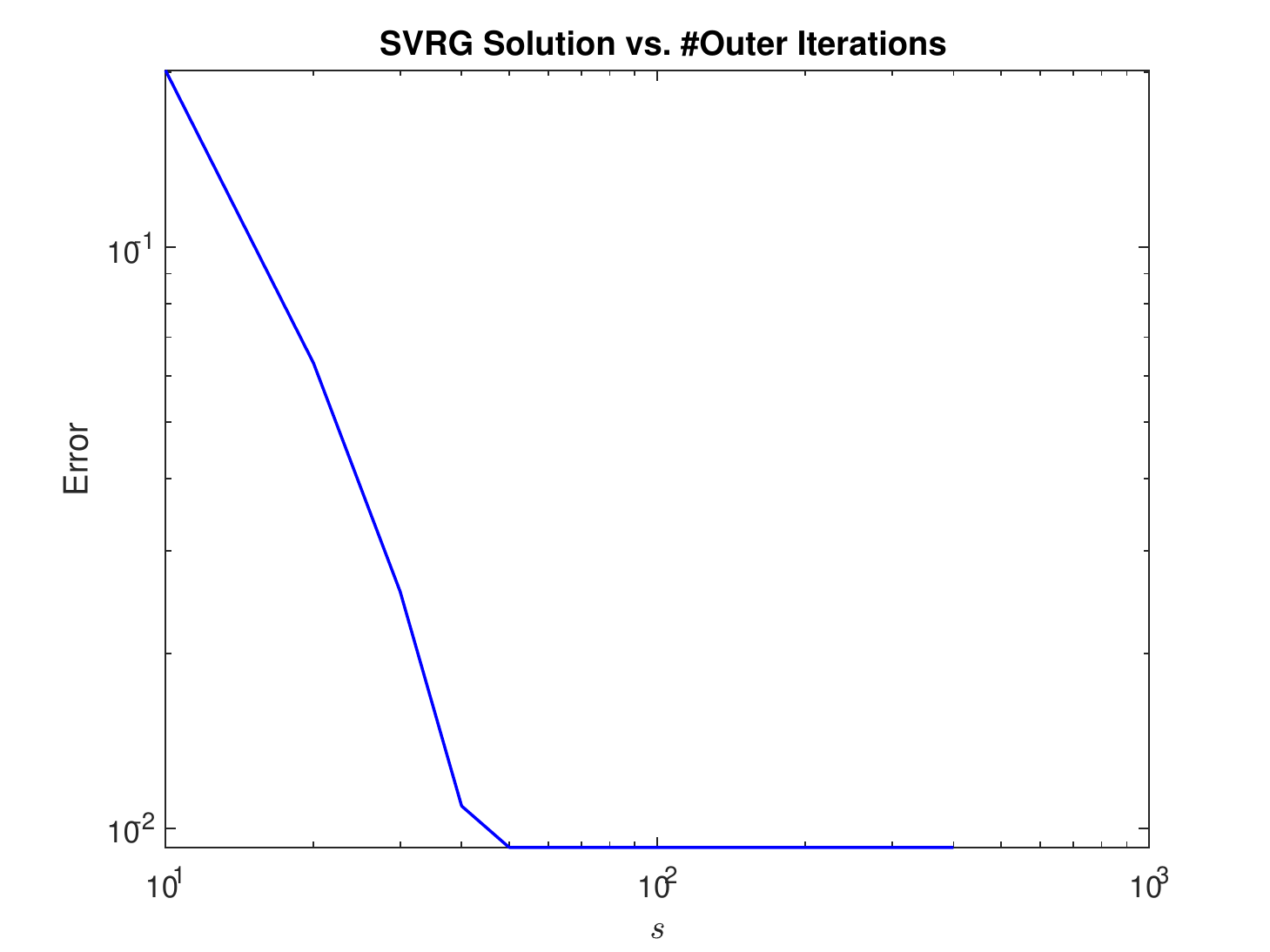}\tabularnewline
\includegraphics[width=0.4\textwidth]{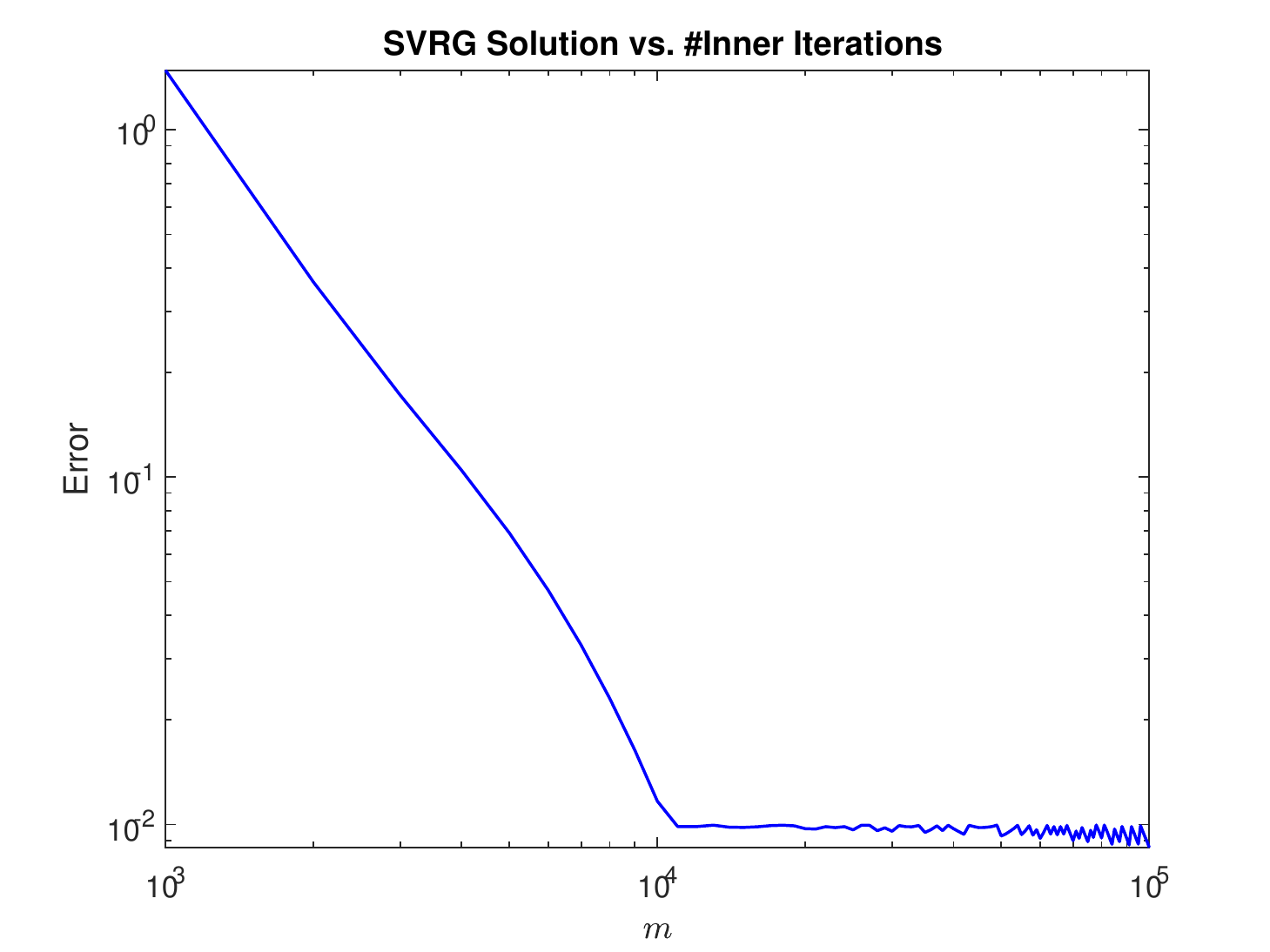} &  & \includegraphics[width=0.4\textwidth]{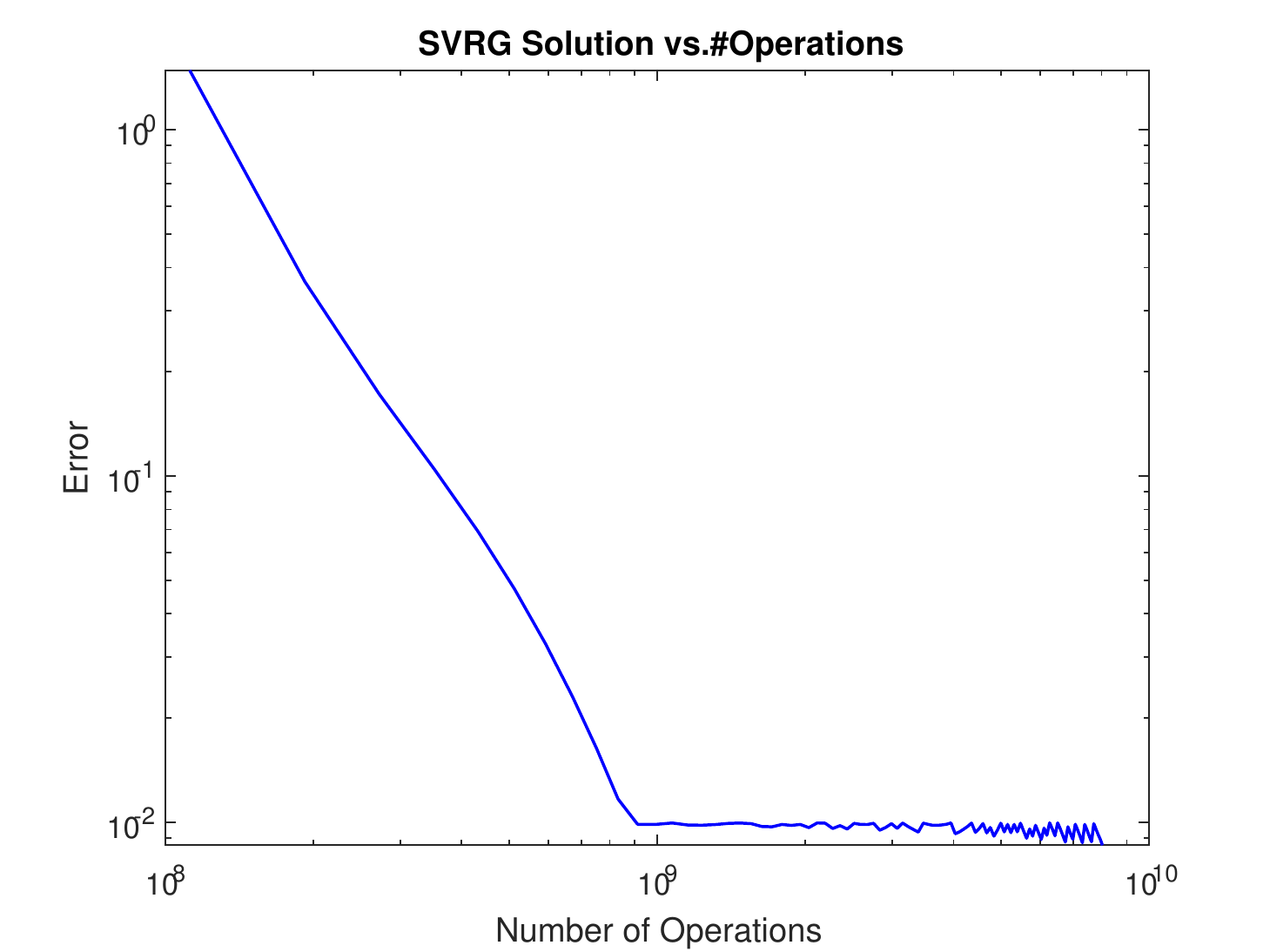}\tabularnewline
\end{tabular}
\par\end{centering}
\caption{\label{fig:SVRG-Test-Error}Experiment with SVRG for KRR. The target
function and the samples are shown in the top left graph. The top
right and bottom left graph show the test error when varying $s$
or $m$ (respectively). In the bottom right we show the test error
as a function of the number of operations.}
\end{figure}

\section{\label{sec:sampling}Converting SILR to Finite Linear Least Squares
via Sampling}

All the previous algorithms we presented for solving SILR problems
either assumed the chebfun model, or relied on the ability to compute
the Gram matrix using an analytic formula (e.g. kernel ridge regression).
In practice, such formulas are not always available, and the chebfun
model is implemented only in software, and even then only when the
columns/rows of the quasimatrix are univariate functions\footnote{While chebfun does support bivariate functions, it does not seem to
support quasimatrices of bivariate functions.}. Thus, a different technique is needed in order to solve SILR problems
that violate these constraints. One natural approach for approximately
solving a SILR problem is to discretize the infinite dimension via
sampling.

For brevity, let us focus on overdetermined SILR (Eq.~(\ref{eq:over-silr})).
In order to discuss sampling, we need a coordinate representation
of $\matA$. Thus, we assume that ${\cal H}=L_{2}(\Omega,d\mu)$ for
some measurable index set $\Omega\subseteq\R^{d}$ and that we have
a coordinate representation $\z_{\matA}:\Omega\to\mathbb{C}^{n}$
for $\matA$ and $\z_{\b}:\Omega\to\mathbb{C}$ for $\b$. A generic
approach is as follows. We first select $s$ coordinates $\veta_{1},\dots,\veta_{s}\in\Omega$,
and associated weights $w_{1},\dots,w_{s}\in\R$. We then form the
row sampled matrix $\matA_{\veta}$ and row sampled vector $\b_{\veta}$
as follows:
\[
\matA_{\veta}=\left[\begin{array}{c}
w_{1}\z_{\matA}(\veta_{1})^{\conj}\\
w_{2}\z_{\matA}(\veta_{2})^{\conj}\\
\vdots\\
w_{s}\z(\veta_{s})^{\conj}
\end{array}\right],\quad\quad\b_{\veta}=\left[\begin{array}{c}
w_{1}\overline{\z_{\b}(\veta_{1})}\\
w_{2}\overline{\z_{\b}(\veta_{2})}\\
\vdots\\
w_{s}\overline{\z_{\b}(\veta_{s})}
\end{array}\right]\,.
\]
We now solve the sampled problem (which is a finite linear least squares
problem):
\[
\tilde{\x}=\arg\min_{\x\in\C^{n}}\TNormS{\matA_{\veta}\x-\b_{\veta}}+\lambda\TNormS{\x}\,.
\]
Solving this sampled problem can be considered as an approximation
to the SILR problem, and as explained in Section~\ref{sec:SILR}
is the scheme used in least squares approximation of functions \cite{cohen2013stability,cohen2017optimal}
and random Fourier features \cite{rahimi2008random}.

To make the method concrete we need to address a couple of related
questions. Given $\veta_{1},\dots,\veta_{s}$, can we relate $\tilde{\x}$
to $\x^{\star}$? How can we select $\veta_{1},\dots,\veta_{s}\in\Omega$
and $w_{1},\dots,w_{s}\in\R$ so that $\tilde{\x}$ is a good enough
approximate solution? Similar questions have been asked, and answered,
for finite linear least squares~\cite{Woodruff14,avron2017faster,ACW17},
and various structural conditions have been suggested. The following
result is similar to ones that appear in the literature on sampling
finite linear least squares problems.
\begin{prop}
\label{prop:part_1}Consider the overdetermined SILR 
\[
\min_{\x\in\C^{n}}\XNormS{\matA\x-\b}{{\cal H}}+\lambda\TNormS{\x}
\]
along with a full rank tall quasimatrix $\matA$ over ${\cal H}$
with $n$ columns and $\lambda\geq0$. Assume that $\left(\XNormS{\matA\x^{\star}-\b}{{\cal H}}+\lambda\TNormS{\x^{\star}}\right)/2\geq\lambda$.
Also assume that we are given a matrix $\matA_{\veta}\in\mathbb{R}^{s\times n}$
and a vector $\b_{\veta}\in\mathbb{R}^{s}$ such that 
\begin{equation}
(1-\epsilon)\left(\XNormS{\matA\x-\b}{{\cal H}}+\lambda\TNormS{\x}+\lambda\right)\leq\TNormS{\matA_{\veta}\x-\b_{\veta}}+\lambda\TNormS{\x}+\lambda\leq(1+\epsilon)\left(\XNormS{\matA\x-\b}{{\cal H}}+\lambda\TNormS{\x}+\lambda\right)\label{eq:over_spectral}
\end{equation}
 for all $\x\in\C^{n}$. Then, 
\[
\XNormS{\matA\tilde{\x}-\b}{{\cal H}}+\lambda\TNormS{\tilde{\x}}\leq\frac{1+2\epsilon}{1-\epsilon}\left(\XNormS{\matA\x^{\star}-\b}{{\cal H}}+\lambda\TNormS{\x^{\star}}\right)\,.
\]
\end{prop}

\begin{proof}
We have
\begin{align*}
\XNormS{\matA\tilde{\x}-\b}{{\cal H}}+\lambda\TNormS{\tilde{\x}} & \leq\frac{1}{1-\epsilon}\left(\TNormS{\matA_{\veta}\tilde{\x}-\b_{\veta}}+\lambda\TNormS{\tilde{\x}}+\lambda\right)-\lambda\\
 & \leq\frac{1}{1-\epsilon}\left(\TNormS{\matA_{\veta}\x^{\star}-\b_{\veta}}+\lambda\TNormS{\x^{\star}}+\lambda\right)-\lambda\\
 & \leq\frac{1+\epsilon}{1-\epsilon}\left(\XNormS{\matA\x^{\star}-\b}{{\cal H}}+\lambda\TNormS{\x^{\star}}+\lambda\right)-\lambda\\
 & \leq\frac{1+2\epsilon}{1-\epsilon}\left(\XNormS{\matA\x^{\star}-\b}{{\cal H}}+\lambda\TNormS{\x^{\star}}\right)
\end{align*}
where the first and third inequalities use Eq.~(\ref{eq:over_spectral}),
the second inequality follows from $\tilde{\x}$ being the minimizer
of the sampled SILR, and the last inequality uses the first assumption.
\end{proof}
Note that Proposition~\ref{prop:part_1} does not require $\matA_{\veta}$
and $\b_{\veta}$ to actually be row samples of $\matA$ and $\b$.

\subsection{\textcolor{black}{Randomized Sampling}}

One approach for selecting $\veta_{1},\dots,\veta_{s}$ and $w_{1},\dots,w_{s}$
is to sample $\veta_{1},\dots,\veta_{s}$ randomly from $\Omega$
and set the weights accordingly. The question is what distribution
on $\Omega$ to use, and how to set the weights? To answer these questions,
we show a general result on the number of samples $s$ required to
ensure Eq.~(\ref{eq:over_spectral}) holds given some distribution
on $\Omega$ and a specific way to set the weights. The result is
based on the concept of ridge leverage scores \cite{el2014fast,cohen2017input},
which we generalize to quasimatrices (the generalization is similar
to the one used in \cite{avron2017random,avron2019universal}).
\begin{defn}
\label{def:lev_scores}Let $\matA$ be a quasimatrix over $L_{2}(\Omega,d\mu)$
equipped with a coordinate representation $\z$, and $\lambda\geq0$.
Further assume that $\mu$ is a probability measure for which a corresponding
density $p$ exists. The \emph{$\lambda$-leverage function }of $\matA$
is 
\[
\tau_{\lambda}:\Omega\to\R,\quad\tau_{\lambda}(\veta)\coloneqq p(\veta)\z(\veta)^{\conj}(\matK+\lambda\matI_{n})^{-1}\z(\veta)
\]
where $\matK=\matA^{*}\matA$ if $\matA$ is a tall quasimatrix, or
$\matK=\matA\matA^{*}$ if $\matA$ is a wide quasimatrix.
\end{defn}

\begin{prop}
[Similar to Proposition 5 in \cite{avron2017random}] Under the same
conditions in Definition \ref{def:lev_scores}:
\[
\int_{\Omega}\tau_{\lambda}(\veta)d\veta=\Trace{(\matK+\lambda\matI_{n})^{-1}\matK}\eqqcolon s_{\lambda}(\matA)\,.
\]
($s_{\lambda}(\matA)$ is called the \emph{statistical dimension }of
$\matA$).
\end{prop}

\begin{lem}
[Similar to Lemma 8 in \cite{avron2017random}]\label{lem:part_2}Consider
the overdetermined SILR\textcolor{red}{{} }
\[
\min\XNormS{\matA\x-\b}{L_{2}(\Omega,d\mu)}+\lambda\TNormS{\x}
\]
where $\matA$ is tall quasimatrix $\matA$ with $n$ columns, and
$\lambda\geq0$. If $\lambda=0$, further assume that $\matA$ is
full rank. Assume we have coordinate representation $\z_{\matA}:\Omega\to\mathbb{C}^{n}$
for $\matA$ and $\z_{\b}:\Omega\to\mathbb{C}$ for $\b$. Assume
that
\[
\left\Vert \left[\begin{array}{c}
\matA^{*}\\
\b^{*}
\end{array}\right]\left[\begin{array}{cc}
\matA & \b\end{array}\right]\right\Vert _{2}\geq\lambda\,.
\]
Let $\tau_{\lambda}(\veta)$ be the $\lambda$-leverage function of
$\left[\begin{array}{cc}
\matA & \b\end{array}\right]$. Let $\tilde{\tau}:\Omega\to\mathbb{R}$ be a measurable function
such that $\tilde{\tau}(\veta)\geq\tau_{\lambda}(\veta)$ for all
$\veta\in\Omega$, and assume that $s_{\tilde{\tau}}=\int_{\Omega}\tilde{\tau}(\veta)d\veta<\infty$.
Also, denote $p_{\tilde{\tau}}(\veta)=\tilde{\tau}(\veta)/s_{\tilde{\tau}}$.
Suppose we sample $\veta_{1},\dots,\veta_{s}$ using $p_{\tilde{\tau}}$
and set $w_{j}=\sqrt{\frac{p(\veta_{j})}{sp_{\tilde{\tau}}(\veta_{j})}}$.
Given $\epsilon\leq1/2$ and $0<\delta<1$, if $s\geq\frac{8}{3}s_{\tilde{\tau}}\epsilon^{-2}\ln(16s_{\lambda}\left(\left[\begin{array}{cc}
\matA & \b\end{array}\right]\right)/\delta)$ then 
\begin{equation}
(1-\epsilon)\left(\XNormS{\matA\x-\b}{L_{2}(\Omega,d\mu)}+\lambda\TNormS{\x}+\lambda\right)\leq\TNormS{\matA_{\veta}\x-\b_{\veta}}+\lambda\TNormS{\x}+\lambda\leq(1+\epsilon)\left(\XNormS{\matA\x-\b}{L_{2}(\Omega,d\mu)}+\lambda\TNormS{\x}+\lambda\right)\label{eq:spectral_lambda}
\end{equation}
 holds with probability of at least $1-\delta$. 
\end{lem}

\begin{proof}
[Proof Sketch.]The proof is very similar to the proof of \cite[Lemma 8]{avron2017random},
so we give only a sketch of the proof. Denote
\[
\hat{\matA}=\left[\begin{array}{c}
\begin{array}{cc}
\matA & \b\end{array}\\
\sqrt{\lambda}\matI_{n+1}
\end{array}\right],\,\hat{\matA}_{\veta}=\left[\begin{array}{c}
\begin{array}{cc}
\matA_{\veta} & \b_{\veta}\end{array}\\
\sqrt{\lambda}\matI_{n+1}
\end{array}\right],\,\hat{\x}=\frac{1}{\sqrt{1+\TNormS{\x}}}\left[\begin{array}{c}
\x\\
-1
\end{array}\right]\,.
\]

Then, the inequality (\ref{eq:spectral_lambda}) is equivalent to
\[
(1-\epsilon)\XNormS{\hat{\matA}\hat{\x}}{L_{2}({\cal X},d\mu)}\leq\TNormS{\hat{\matA}_{\veta}\hat{\x}}\leq(1+\epsilon)\XNormS{\hat{\matA}\hat{\x}}{L_{2}({\cal X},d\mu)}
\]
i.e.,
\[
-\epsilon\hat{\matA}^{*}\hat{\matA}\preceq\hat{\matA}_{\veta}^{*}\hat{\matA}_{\veta}-\hat{\matA}^{*}\hat{\matA}\preceq\epsilon\hat{\matA}^{*}\hat{\matA}\,.
\]
We write $\hat{\matA}^{*}\hat{\matA}=\matV^{*}\ensuremath{\mat{\Sigma}}^{2}\matV$.
The claim is now equivalent to 
\[
-\epsilon\matI_{d}\preceq\mat{\Sigma}^{-1}\matV^{\star}\hat{\matA}_{\veta}^{*}\hat{\matA}_{\veta}\matV\mat{\Sigma}^{-1}-\matI_{d}\preceq\epsilon\matI_{d}\,.
\]
Notice that 
\[
\matA_{\veta}^{*}\matA_{\veta}=\sum_{j=1}^{s}w_{j}^{2}\z_{\matA}(\veta_{j})\z_{\matA}(\veta_{j})^{*},\,\b_{\veta}^{*}\b_{\veta}=\sum_{j=1}^{s}w_{j}^{2}\z_{\b}(\veta_{j})\z_{\b}(\veta_{j})^{*}\,.
\]
It can be seen that $\z(\veta)=\left[\begin{array}{c}
\z_{\matA}(\veta)\\
\z_{\b}(\veta)
\end{array}\right]$ is a coordinate representation for the quasimatrix part of $\hat{\matA}$.
Let
\[
\matS_{j}=\frac{p(\veta_{j})}{p_{\tilde{\tau}}(\veta_{j})}\mat{\Sigma}^{-1}\matV^{*}\z(\veta_{j})\z(\veta_{j})^{*}\matV\mat{\Sigma}^{-1}\,.
\]
It is possible to show that $\mathbb{E}[\matS_{j}^{2}]\preceq s_{\tilde{\tau}}\mathbb{E}[\matS_{j}]$
and $\Trace{\mathbb{E}[\matS_{j}]}=s_{\tilde{\tau}}\cdot s_{\lambda}(\hat{\matK})$.
The claim follows from \cite[Corollary 7.3.3]{tropp2015introduction}.
\end{proof}
A similar result appears in \cite{cohen2017optimal} for truncated
and conditioned least squares approximations of functions, however
without any ridge term. The ridge leverage function can be viewed
as a variant of the Christoffel function \cite{pauwels2018relating}
from the literature on orthogonal polynomials and approximation theory
\cite{pauwels2018relating,nevai1986geza,totik2000asymptotics,borwein2012polynomials}.

One natural strategy for selecting the $\veta_{1},\dots,\veta_{s}$
is to sample them using the distribution $\mu$. We call this strategy
``\emph{natural sampling''. }Using Lemma \ref{lem:part_2} we can
give a bound on the number of samples needed when sampling $\veta_{1},\dots,\veta_{s}$
using this strategy and setting all the weights to $\sqrt{1/s}$.
\begin{prop}
\label{prop:randomized-sampling}Let $\tau_{\lambda}(\veta)$ be the
$\lambda$-leverage function of $\left[\begin{array}{cc}
\matA & \b\end{array}\right]$. Suppose that $M_{\lambda}=M_{\lambda}\left(\left[\begin{array}{cc}
\matA & \b\end{array}\right]\right)\coloneqq\sup_{\veta\in\Omega}\tau_{\lambda}(\veta)/p(\veta)$ is finite. Suppose we sample $\veta_{1},\dots,\veta_{s}$ using $\mu$,
and set $w_{j}=\sqrt{1/s}$ for $j=1,\dots,s$. If 
\[
s\geq\frac{8}{3}M_{\lambda}\epsilon^{-2}\ln(16s_{\lambda}\left(\left[\begin{array}{cc}
\matA & \b\end{array}\right]\right)/\delta)
\]
then Eq.~(\ref{eq:spectral_lambda}) holds with probability of at
least $1-\delta$.
\end{prop}

\begin{proof}
Let us define $\tilde{\tau}(\veta)=M_{\lambda}p(\veta)$. Notice that
$s_{\tilde{\tau}}=M_{\lambda}$ and that $p_{\tilde{\tau}}(\veta)=p(\veta)$.
Thus, the conditions of Proposition~\ref{lem:part_2} hold if we
sample using $p(\cdot)$ and set the weights to $\sqrt{1/s}$, and
the claim follows.
\end{proof}
The quantity $M_{\lambda}$ is a generalization of the concept of
matrix coherence \cite{AMT10} to quasimatrices. A similar quantity
appears in \cite{cohen2013stability} in the context of function approximation
using sampling. When using natural sampling, the number of samples
required for Eq.~(\ref{eq:spectral_lambda}) to hold with high probability
depends on the coherence of the quasimatrix, which can be large. Sampling
using the ridge leverage scores, often referred to as \emph{leverage
score sampling}, yields a better bound since $s_{\lambda}([\begin{array}{cc}
\matA & \b\end{array}])\leq M_{\lambda}([\begin{array}{cc}
\matA & \b\end{array}])$.

Of course, it is not simple to sample using the ridge leverage function.
Cohen and Migliorati suggested a method from leverage score sampling
when $\lambda=0$ \cite{cohen2017optimal}. Their method is based
on sequential conditional sampling, where individual coordinates are
sampled using either rejection sampling or inversion transform sampling.
An alternative approach is to find some simple and easy way to sample
upper bound on $\tau_{\lambda}$. For this to be worthwhile, the bound
has to be tighter than the bound $\tau_{\lambda}(\veta)\leq M_{\lambda}p(\veta)$
used in Proposition~\ref{prop:randomized-sampling}. This approach
is used in \cite{avron2017random,avron2019universal}.

\subsection{Quadrature Sampling}

In this section, we discuss deterministic sampling using quadrature
formulas. For simplicity, we assume that $\Omega=[-1,1]$ and that
$\mu$ is the Lebesgue measure on $[-1,1].$ Accordingly, the sampling
scheme is based on the Gauss-Legendre quadrature. Higher dimensional
domains can be handled via tensoring the quadrature. We also assume
that $\lambda>0$. Let $\z_{\matA}:\R\to\C^{n}$ be a coordinate representation
of $\matA$. We can write 
\[
\matA^{*}\matA=\int_{-1}^{1}\z_{\matA}(\eta)\z_{\matA}(\eta)^{\conj}d\eta\,.
\]
Furthermore, for every $\x\in\mathbb{R}^{n}$
\[
\XNormS{\matA\x}{L_{2}([-1,1],d\mu)}=\int_{-1}^{1}\x^{\T}\z_{\matA}(\eta)\z_{\matA}(\eta)^{\conj}\x d\eta=\int_{-1}^{1}\left|\z_{\matA}(\eta)^{\conj}\x\right|^{2}d\eta\,.
\]
Let $\z_{\b}:\R\to\C$ be a coordinate representation of $\b$. Then,
\[
\matA^{*}\b=\int_{-1}^{1}\z_{\matA}(\eta)\overline{\z_{\b}(\eta)}d\eta,\quad\XNormS{\b}{L_{2}([-1,1],d\mu)}=\int_{-1}^{1}\left|\z_{\b}(\eta)\right|^{2}d\eta\,.
\]
We conclude that the overdetermined SILR can be written as an integral
form
\begin{equation}
\int_{-1}^{1}f_{\x}(\eta)d\eta=\XNormS{\matA\x-\b}{L_{2}([-1,1],d\mu)}+\lambda\TNormS{\x}\label{eq:int_A}
\end{equation}
where
\[
f_{\x}(\eta)=\left|\z_{\matA}(\eta)^{\conj}\x-\z_{\b}(\eta)\right|^{2}+\frac{\lambda}{2}\TNormS{\x}\,.
\]

The underlying idea is to approximate the integral in Eq.~(\ref{eq:int_A})
using the Gauss-Legendre quadrature. For a given $\epsilon\in(0,1)$,
our algorithm sets the nodes $\eta_{1},\dots,\eta_{s}\in[-1,1]$ to
be the Gauss-Legendre quadrature nodes, and sets the weights $w_{1},\dots,w_{s}>0$
so that their square are the Gauss-Legendre quadrature weights. We
set $s$ to be large enough so that 
\begin{equation}
\frac{\left|\int_{-1}^{1}f_{\x}(\eta)d\eta-\sum_{j=1}^{s}w_{j}^{2}f_{\x}(\eta_{j})\right|}{\int_{-1}^{1}f_{\x}(\eta)d\eta+\lambda}\leq\epsilon\,.\label{eq:quad}
\end{equation}
Once $\eta_{1},\dots,\eta_{s}$ and the weights $w_{1},\dots,w_{s}$
are computed, we can define $\matA_{\veta}$ and $\b_{\veta}$ as
before. We have
\begin{align*}
\sum_{j=1}^{s}w_{j}^{2}f_{\x}(\eta_{j}) & =\TNormS{\matA_{\veta}\x-\b_{\veta}}+\lambda\TNormS{\x}
\end{align*}
so if Eq.~(\ref{eq:quad}) holds then Eq.~(\ref{eq:over_spectral})
holds (with $\H=L_{2}([-1,1],d\mu)$), and we can apply Proposition~\ref{prop:part_1}.

To determine how many quadrature nodes $s$ are needed so that Eq.~(\ref{eq:quad})
holds, we can \textcolor{black}{apply the following theorem, which
is a modified version of \cite[Theorem 11]{shustin2021gauss} for
the function $g_{\x}(\eta)\coloneqq f_{\x}(\eta)/(\int_{-1}^{1}f_{\x}(\eta)d\eta+\lambda)$}.
Since the proof is a simple modification of the proof\textcolor{black}{{}
\cite[Theorem 11]{shustin2021gauss},} we omit it.
\begin{thm}
\label{thm:quad}Let $E$ be the (Bernstein) ellipse in the complex
plane with foci $\pm1$ that passes through $i$, and let $\rho=1+\sqrt{2}$.
Assume that both real and imaginary parts of $\z_{\matA}(\cdot)_{i},\,i=1,\dots,n$
and $\z_{\b}(\cdot)$ are analytic on $\mathbb{R}$, and denote their
analytic continuations by $\hat{\z}_{\matA}(\cdot)$ and $\hat{\z}_{\b}(\cdot)$
correspondingly. Denote
\[
M_{\mat A}\coloneqq\sup_{\eta\in E}\InfNorm{\hat{\z}_{\matA}(\eta)},\quad M_{\b}\coloneqq\sup_{\eta\in E}\InfNorm{\hat{\z}_{\b}(\eta)}\,.
\]
Then, given a small $\epsilon$, for
\[
s\geq\frac{\ln\left(8(\lambda^{-1}(nM_{\matA}^{2}+M_{\b}^{2})+1)\right)-\ln\epsilon-\ln\sqrt{2}}{2\ln(1+\sqrt{2})}+1
\]
we have
\[
\left|\int_{-1}^{1}g_{\x}(\eta)d\eta-\sum_{j=1}^{s}w_{j}g_{\x}(\eta_{j})\right|\leq\epsilon
\]
where $\eta_{1},\dots\eta_{s}$ are chosen to be the Gauss-Legendre
quadrature nodes, and $w_{1}^{2},\dots,w_{s}^{2}$ are the Gauss-Legendre
quadrature weights.
\end{thm}

\begin{rem}
For $\eta\in E$, we denote $c_{\x}=\XNormS{\matA\x-\b}{L_{2}([-1,1],d\mu)}+\lambda\TNormS{\x}+\lambda$
and bound $g_{\x}(\eta)$ as follows{\footnotesize{}
\begin{align*}
g_{\x}(\eta) & =\frac{1}{c_{\x}}\left(\left|\z_{\matA}(\eta)^{\conj}\x-\z_{\b}(\eta)\right|^{2}+\frac{\lambda}{2}\TNormS{\x}\right)\\
 & =\frac{1}{c_{\x}}\left(\left|\left[\begin{array}{cc}
\z_{\matA}(\eta)^{\conj} & \z_{\b}(\eta)\end{array}\right]\left[\begin{array}{c}
\x\\
-1
\end{array}\right]\right|^{2}+\frac{\lambda}{2}\TNormS{\x}\right)\\
 & =\frac{1}{c_{\x}}\left(\left|\left[\begin{array}{cc}
\z_{\matA}(\eta)^{\conj} & \z_{\b}(\eta)\end{array}\right]\hat{\matK}^{-1/2}\hat{\matK}^{1/2}\left[\begin{array}{c}
\x\\
-1
\end{array}\right]\right|^{2}+\frac{\lambda}{2}\TNormS{\x}\right)\\
 & \leq\frac{1}{c_{\x}}\left(\left[\begin{array}{cc}
\z_{\matA}(\eta)^{\conj} & \z_{\b}(\eta)\end{array}\right]\hat{\matK}^{-1}\left[\begin{array}{c}
\z_{\matA}(\eta)\\
\z_{\b}(\eta)^{*}
\end{array}\right]\cdot\left(\left[\begin{array}{cc}
\x & -1\end{array}\right]\hat{\matK}\left[\begin{array}{c}
\x\\
-1
\end{array}\right]\right)+\frac{\lambda}{2}\TNormS{\x}\right)\\
 & =\left[\begin{array}{cc}
\z_{\matA}(\eta)^{\conj} & \z_{\b}(\eta)\end{array}\right]\hat{\matK}^{-1}\left[\begin{array}{c}
\z_{\matA}(\eta)\\
\z_{\b}(\eta)^{*}
\end{array}\right]+\frac{1}{c_{\x}}\cdot\frac{\lambda}{2}\TNormS{\x}\\
 & \leq\lambda^{-1}(\|\z_{\matA}(\eta)\|_{2}^{2}+\left|\z_{\b}(\eta)\right|^{2})+\frac{1}{2}\\
 & \leq\lambda^{-1}(nM_{\matA}^{2}+M_{\b}^{2})+\frac{1}{2}
\end{align*}
}where $\hat{\matK}=\left[\begin{array}{cc}
\matA^{*}\matA & \matA^{*}\b\\
\b^{*}\matA & \b^{*}\b
\end{array}\right]+\lambda\matI_{n+1}$ and in the first inequality we use the Cauchy-Schwarz inequality.
Theorem \ref{thm:quad} yields
\[
\frac{\left|\int_{-1}^{1}f_{\x}(\eta)d\eta-\sum_{j=1}^{s}w_{j}^{2}f_{\x}(\eta_{j})\right|}{\int_{-1}^{1}f_{\x}(\eta)d\eta+\lambda}=\left|\int_{-1}^{1}g_{\x}(\eta)d\eta-\sum_{j=1}^{s}w_{j}^{2}g_{\x}(\eta_{j})\right|\leq\epsilon\,.
\]
\end{rem}

We can generalize the above theorem, which is specific for $\Omega=[-1,1]$,
to complex sets and/or high dimensional sets with a variety of probability
measures on them, as done in \cite{shustin2021gauss}.

\subsection{Numerical Example}

We illustrate both sampling approaches, randomized and quadrature,
on a small numerical example. Consider trying to approximate on $[-1,1]$
the Runge function using polynomial of degree 39. We use the Chebyshev
basis, i.e.
\[
\matA=\left[\begin{array}{cccc}
T_{0} & T_{1} & \dots & T_{39}\end{array}\right],\quad\b=\left[\frac{1}{1+25x^{2}}\right]
\]
with $\lambda=10^{-4}$.

The leftmost graph in Figure~\ref{fig:sampling_errors} shows the
ridge leverage density of $\left[\begin{array}{cc}
\matA & \b\end{array}\right]$, and compares it to the uniform density. We also plot the density
of the limiting distribution of Legendre nodes. We see very close
alignment between the ridge leverage score density and the density
of the Legendre nodes. We note that in this case $s_{\lambda}=39.99$.
In contrast $M_{\lambda}=798.28$, and thus we will need about 95\%
less samples when using leverage score sampling when compared to natural
sampling. However, even for $\epsilon=0.01$, the number of samples
required for randomized sampling is huge. In contrast, for $\epsilon=0.01$
only $s=73$ features are required using quadrature features. Nevertheless,
in the experiments we use $s=100$ for both randomized and quadrature
sampling.

The middle and rightmost graph in Figure~\ref{fig:sampling_errors}
shows the function approximation (on the left), and the error in approximating
the function (on the right). We use both natural sampling and leverage
score sampling, where we used inverse transform sampling for leverage
score sampling. With $s=100$, using quadrature sampling and leverage
score sampling we get small errors: the maximum absolute error is
$4.48\times10^{-4}$ for quadrature sampling, and $9.82\times10^{-4}$
for leverage score sampling. Natural sampling has large error near
the boundary of $[-1,1]$ (as expected), and the maximum absolute
error is $0.0581$.

\begin{figure}[H]
\begin{centering}
\begin{tabular}{ccc}
\includegraphics[width=0.3\textwidth]{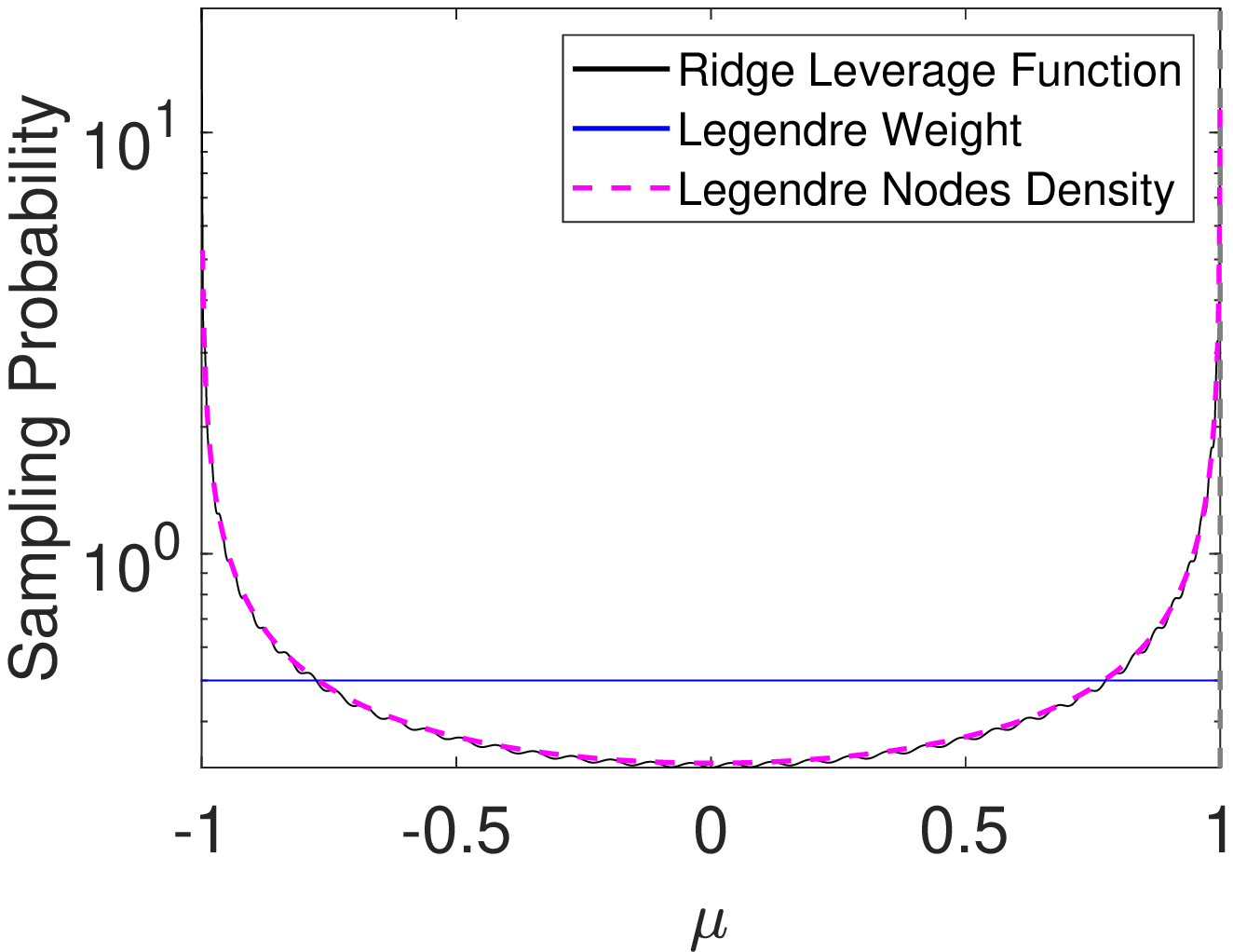} & \includegraphics[width=0.3\textwidth]{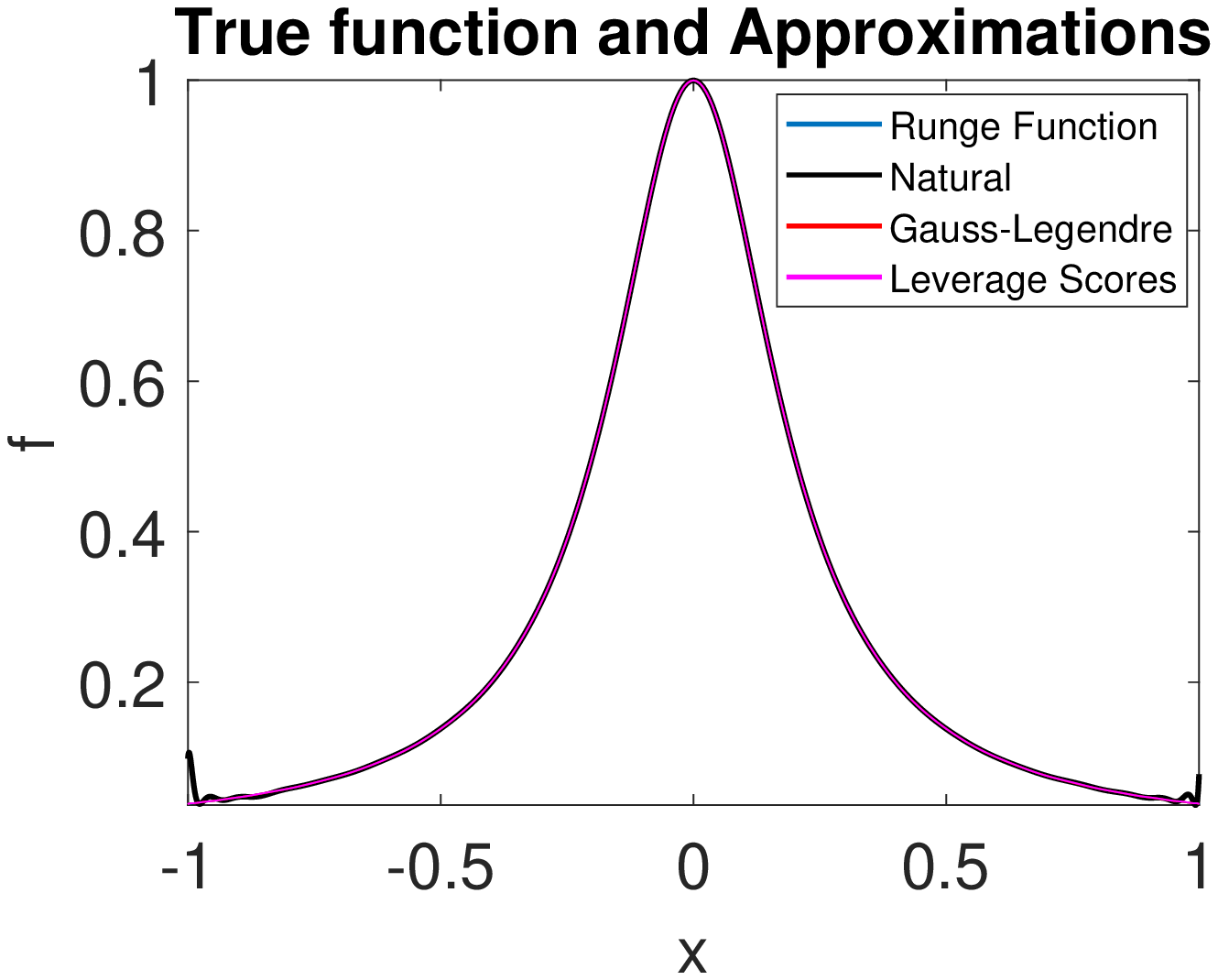} & \includegraphics[width=0.3\textwidth]{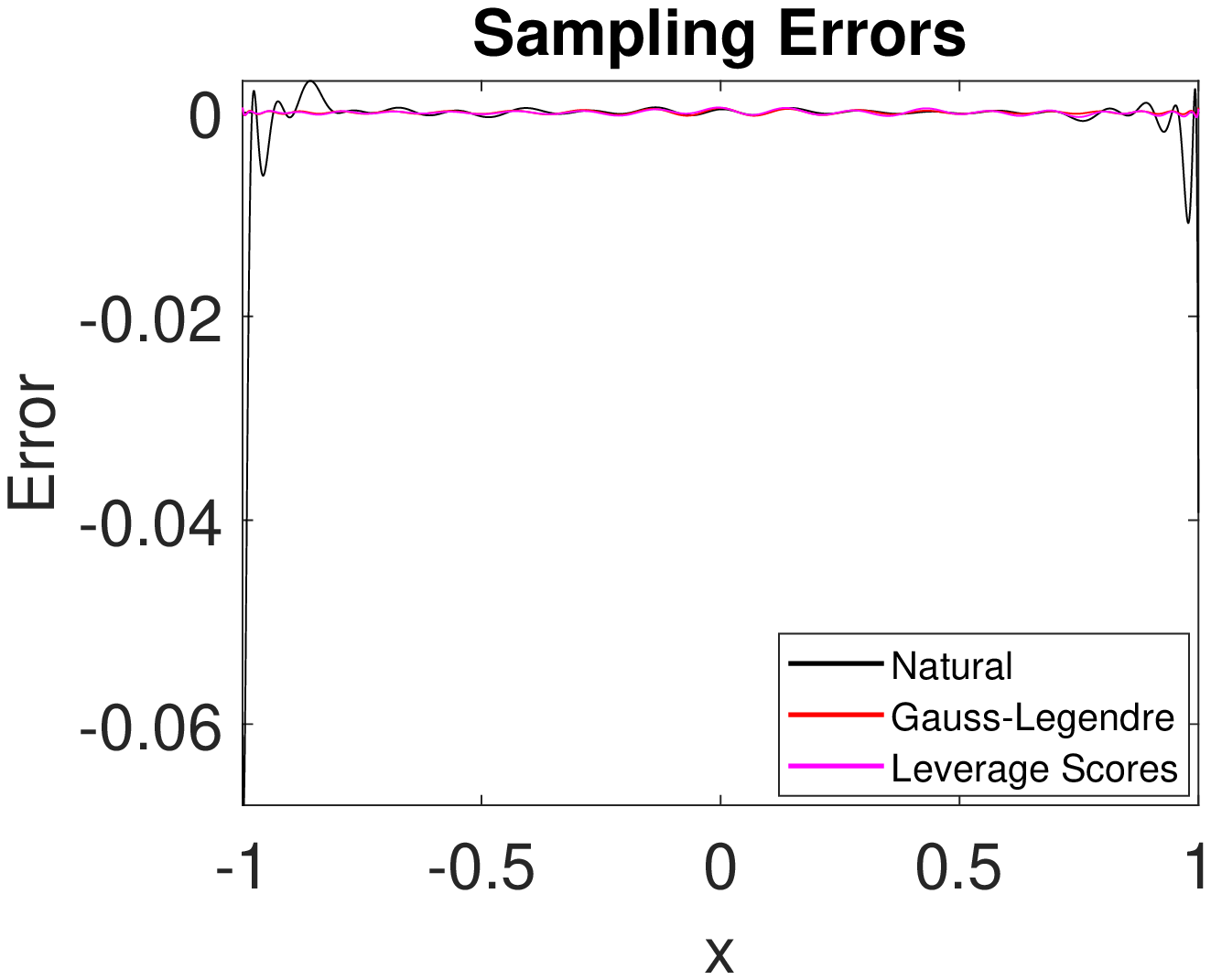}\tabularnewline
\end{tabular}
\par\end{centering}
\caption{\label{fig:sampling_errors} Numerical illustration: approximating
the Runge function using various sampling methods.}
\end{figure}

\section{Conclusions and Future Work}

In this paper, we gave an algebraic framework for working with quasimatrices
and explored the use of this framework to solve semi-infinite linear
regression problems, i.e. regression problems where the system's matrix
has an infinite amount of rows or columns. We discussed various applications,
such as function approximation and supervised learning (using kernel
ridge regression). We offered several classes of algorithms for solving
SILR problems: direct methods, iterative methods (generalizing known
iterative methods such as LSMR as an example of a Krylov subspace
method and SVRG as an example of a stochastic optimization method).
Finally, motivated by recent research on randomized numerical linear
algebra methods for solving finite linear least squares problems,
we explored the use of sampling techniques to approximate the solution
of a SILR, where sampling can be either randomized or deterministic.
Possible future directions are to further leverage advanced randomized
linear algebra methods, such as sketching, whereas the main challenge
is in how to generate a random quasimatrix from the correct distribution.
Another interesting idea is to generalize the Batson-Spielman-Srivastava
(BSS) process for iteratively building a spectral approximation of
a matrix using columns samples~\cite{BSS14} to quasimatrices.

\subsection*{Acknowledgements}

This research was supported by BSF grant 2017698.

\bibliographystyle{plain}
\bibliography{quasi}

\appendix

\section{\label{sec:SVRG_int}Stochastic Variance Reduced Gradient with Integrable
Sums}

The usual SVRG algorithm \cite{johnson2013accelerating,xiao2014proximal}
is defined for objective functions that have finite sum structure,
i.e. 
\begin{equation}
f(\x)=\frac{1}{n}\sum_{i=1}^{n}f_{i}(\x)\,.\label{eq:finite-sum}
\end{equation}
Here we propose a variant of the algorithm designed for objective
functions that can be written as an integral. Let $\mu$ be some probability
measure on a measurable index set, $\Omega$. Our variant of SVRG
is designed for functions than can be written as
\begin{equation}
f(\x)=\int_{\Omega}f_{\veta}(\x)d\mu(\veta)\label{eq:int-sum}
\end{equation}
where the integral should be interepted as a Lebesgue integral. Notice
that Eq.~(\ref{eq:finite-sum}) is a special case of Eq.~(\ref{eq:int-sum}):
$\Omega=\left\{ 1,\ldots,n\right\} $ and $\mu(A)=\frac{|A|}{n}$.
The proposed algorithm is summarized in Algorithm~\ref{alg:svrg_int}.

\begin{algorithm}[H]
\begin{algorithmic}[1]

\STATE \textbf{Inputs: }initial\textbf{ $\tilde{\x}_{0}$}, learning
rate $\alpha$, frequency $m$

\STATE \textbf{Iterate: }for $s=1,2,\ldots$

\STATE $\qquad$$\tilde{\x}=\tilde{\x}_{s-1}$

\STATE $\qquad$$\tilde{\mu}=\nabla\int_{\Omega}f_{\veta}(\tilde{\x})d\mu(\veta)=\nabla f(\tilde{\x})$

\STATE $\qquad$$\x_{0}=\tilde{\x}$

\STATE $\qquad$\textbf{Iterate: }for $k=1,2,\ldots,m$

\STATE $\qquad\qquad$sample $\veta_{k}$ according to the probability
of $\veta$ and update

\STATE $\qquad\qquad$$\x_{k}=\x_{k-1}-\alpha\left(\nabla f_{\veta_{k}}(\x_{k-1})-\nabla f_{\veta_{k}}(\tilde{\x})+\tilde{\mu}\right)$

\STATE \textbf{$\qquad$end}

\STATE \textbf{option I: }set $\tilde{\x}_{s}=\x_{m}$

\STATE \textbf{option II: }set $\tilde{\x}_{s}=\frac{1}{m}\sum_{k=1}^{m}\x_{k}$

\STATE \textbf{end}

\end{algorithmic}

\caption{\label{alg:svrg_int}SVRG for integrable objective functions.}
\end{algorithm}

As in common convex optimization, certain assumptions must be made
in order for the algorithm to converge. We prove that Algorithm \ref{alg:svrg_int}
converges and analyze the convergence rate, when the following assumptions
hold. We start with assumptions that are analogous to the assumptions
in finite sum SVRG, which we already mentioned in Section \ref{subsec:SVRG_SILR}.
\begin{assumption}
\label{assu:Lip-eta-appendix}For all $\veta\in\Omega$, $\nabla f_{\veta}(\x)$
is Lipschitz continuous, i.e., there exists $L_{\veta}>0$ such that
for all $\x,\y\in\mathbb{R}^{n}$ 
\[
\Vert\nabla f_{\veta}(\x)-\nabla f_{\veta}(\y)\Vert\leq L_{\veta}\Vert\x-\y\Vert\,.
\]
\end{assumption}

\begin{assumption}
\label{assu:strong-convex-appendix}Suppose that $f(\x)$ is strongly
convex, i.e., there exist $\gamma>0$ such that for all $\x,\y\in\mathbb{R}^{n}$
\[
f(\x)-f(\y)\geq\frac{\gamma}{2}\TNormS{\x-\y}+\nabla f(\boldsymbol{y})^{\T}(\x-\y)\,.
\]
\end{assumption}

Next, we list assumptions that trivially hold for the finite case
but are required for the continuous case.
\begin{assumption}
\label{assu:grad-int-appendix}The equality $\nabla f(\x)=\int_{\Omega}\nabla f_{\veta}(\x)d\mu(\veta)$
hold.
\end{assumption}

Suppose $\Omega=\mathbb{R}^{d}$ and $f_{\veta}(\x),\nabla f_{\veta}(\x)\in L_{1}(\Omega)$
with respect to $\eta$. Then, Assumption \ref{assu:grad-int-appendix}
holds from Leibniz integral rule.
\begin{assumption}
\label{assu:finite-L_eta-appendix}$L_{\sup}\coloneqq\sup_{\veta\in\Omega}L_{\veta}<\infty$.
\end{assumption}

Assumptions \ref{assu:grad-int-appendix} and \ref{assu:finite-L_eta-appendix}
imply that $\nabla f(\x)$ is Lipschitz continuous with Lipschitz
constant $L\leq L_{\sup}$. Note that for the finite sum case, Assumptions~\ref{assu:grad-int-appendix}
and \ref{assu:finite-L_eta-appendix} hold trivially, but this is
no longer the case in the integrable case.
\begin{cor}
If Assumptions \ref{assu:grad-int-appendix},\ref{assu:finite-L_eta-appendix}
hold, then we can make Assumptions \ref{assu:Lip-eta-appendix},\ref{assu:strong-convex-appendix}
hold for the continuous case.
\end{cor}

We now analyze Algorithm~\ref{assu:grad-int-appendix}. The analysis
follows the analysis in \cite{johnson2013accelerating,xiao2014proximal}
quite closely, making adjustments where necessary for integrals instead
of sums, and using the additional assumptions when needed.
\begin{lem}
\label{lem:L2_grad}Suppose Assumptions \ref{assu:Lip-eta-appendix},\ref{assu:grad-int-appendix},\ref{assu:finite-L_eta-appendix}
hold. Let $\x^{\star}=\arg\,\min_{\x}f(\x)$ and $L_{\sup}=\sup_{\veta\in\Omega}L_{\veta}$.
Then
\[
\int_{\Omega}\TNormS{\nabla f_{\veta}(\x)-\nabla f_{\veta}(\x^{\star})}d\mu(\veta)\leq2L_{\sup}\left(f(\x)-f(\x^{\star})\right)\,.
\]
\end{lem}

\begin{proof}
Given any $\veta\in\Omega$, let 
\[
g_{\veta}(\x)=f_{\veta}(\x)-f_{\veta}(\x^{\star})-\nabla f_{\veta}(\x^{\star})^{T}(\x-\x^{\star})\,.
\]
It can be seen that $\nabla g_{\veta}(\x^{\star})=0$, and hence $\x^{\star}=\arg\min_{\x}g_{\veta}(\x)$.
Moreover, from Assumption \ref{assu:Lip-eta-appendix}, $\nabla g_{\veta}(\x)=\nabla f_{\veta}(\x)-\nabla f_{\veta}(\x^{\star})$
is Lipschitz continuous with constant $L_{\veta}$. This yields

\[
g_{\veta}(\x)-g_{\veta}(\y)\leq\frac{L_{\veta}}{2}\TNormS{\x-\y}+\nabla g_{\veta}(\boldsymbol{y})^{\T}(\x-\y)
\]
for any $\x,\y\in\mathbb{R}^{n}$ (see \cite[Lemma 1.2.3]{nesterov2003introductory}).
Replacing $\x$ with $\x-\frac{1}{L_{\veta}}\nabla g_{\veta}(\x)$
and $\y$ with $\x$, gives
\[
g_{\veta}\left(\x-\frac{1}{L_{\veta}}\nabla g_{\veta}(\x)\right)\leq g_{\veta}(\x)-\frac{1}{2L_{\veta}}\TNormS{\nabla g_{\veta}(\x)}\,.
\]
Since $\min_{\x}g_{\veta}(\x)=g_{\veta}(\x^{\star})=0$, we have $0\leq g_{\veta}(\x-\nabla g_{\veta}(\x)/L_{\veta})$,
which implies
\begin{equation}
\frac{1}{2L_{\veta}}\TNormS{\nabla g_{\veta}(\x)}\leq g_{\veta}(\x)\,.\label{eq:Lip_result_g}
\end{equation}
Substituting the definition of $g$ gives 
\[
\TNormS{\nabla f_{\veta}(\x)-\nabla f_{\veta}(\x^{\star})}\leq2L_{\veta}\left(f_{\veta}(\x)-f_{\veta}(\x^{\star})-\nabla f_{\veta}(\x^{\star})^{T}(\x-\x^{\star})\right)\,.
\]

Now, by taking an integral over $\Omega$, we have

\begin{align*}
\int_{\Omega}\TNormS{\nabla f_{\veta}(\x)-\nabla f_{\veta}(\x^{\star})}d\mu(\veta) & \leq2L_{\sup}\int_{\Omega}f_{\veta}(\x)-f_{\veta}(\x^{\star})-\nabla f_{\veta}(\x^{\star})^{\T}(\x-\x^{\star})d\mu(\veta)\\
 & =2L_{\sup}\left(f(\x)-f(\x^{\star})-\nabla f(\x^{\star})^{\T}(\x-\x^{\star})\right)=2L_{\sup}\left(f(\x)-f(\x^{\star})\right)
\end{align*}
where in the first inequality we use Assumption \ref{assu:finite-L_eta-appendix},
in the second equality we use Assumption \ref{assu:grad-int-appendix}
and the last equality is due to the fact that $\nabla f(\x^{\star})=0$.
\end{proof}
\begin{cor}
\label{cor:grad_differ}Denote $\v_{k}=\nabla f_{\veta_{k}}(\x_{k-1})-\nabla f_{\veta_{k}}(\tilde{\x})+\tilde{\mu}$.
Then, conditioned on $\x_{k-1}$ we have 
\[
\mathbb{E}\TNormS{\v_{k}}\leq4L_{\sup}\left(f(\x_{k-1})-2f(\x^{\star})+f(\tilde{\x})\right)\,.
\]
\end{cor}

\begin{proof}
Conditioned on $\x_{k-1}$, taking expectation with respect to $\veta_{k}$
gives $\mathbb{E}\left[\nabla f_{\eta_{k}}(\x_{k-1})\right]=\nabla f(\x_{k-1})$.
Similarly, $\mathbb{E}\left[\nabla f_{\veta_{k}}(\tilde{\x})\right]=\nabla f(\tilde{\x})$.
Therefore

\begin{equation}
\mathbb{E}\left[\v_{k}\right]=\mathbb{E}\left[\nabla f_{\veta_{k}}(\x_{k-1})-\nabla f_{\veta_{k}}(\tilde{\x})+\tilde{\mu}\right]=\nabla f(\x_{k-1})\,.\label{eq:vk_grad}
\end{equation}

Now,
\begin{align*}
\mathbb{E}\TNormS{\v_{k}} & =\mathbb{E}\TNormS{\nabla f_{\veta_{k}}(\x_{k-1})-\nabla f_{\veta_{k}}(\tilde{\x})+\tilde{\mu}+\nabla f_{\veta_{k}}(\x^{\star})-\nabla f_{\veta_{k}}(\x^{\star})}\\
 & \leq2\mathbb{E}\TNormS{\nabla f_{\veta_{k}}(\x_{k-1})-\nabla f_{\veta_{k}}(\x^{\star})}+2\mathbb{E}\TNormS{\nabla f_{\veta_{k}}(\x^{\star})-\nabla f_{\veta_{k}}(\tilde{\x})+\tilde{\mu}}\\
 & =2\mathbb{E}\TNormS{\nabla f_{\veta_{k}}(\x_{k-1})-\nabla f_{\veta_{k}}(\x^{\star})}+2\mathbb{E}\TNormS{\nabla f_{\veta_{k}}(\tilde{\x})-\nabla f_{\veta_{k}}(\x^{\star})-\mathbb{E}\left[\nabla f_{\veta_{k}}(\tilde{\x})-\nabla f_{\veta_{k}}(\x^{\star})\right]}\\
 & \leq2\mathbb{E}\TNormS{\nabla f_{\veta_{k}}(\x_{k-1})-\nabla f_{\veta_{k}}(\x^{\star})}+2\mathbb{E}\TNormS{\nabla f_{\veta_{k}}(\tilde{\x})-\nabla f_{\veta_{k}}(\x^{\star})}\\
 & \leq4L_{\sup}\left(f(\x_{k-1})-2f(\x^{\star})+f(\tilde{\x})\right)
\end{align*}

where in the first inequality we use $\Vert a+b\Vert^{2}\leq2(\Vert a\Vert^{2}+\Vert b\Vert^{2})$.
The second equality uses $\mathbb{E}\left[\nabla f_{\veta_{k}}(\tilde{\x})\right]=\tilde{\mu},\,\,\mathbb{E}\left[\nabla f_{\veta_{k}}(\x^{\star})\right]=\nabla f(\x^{\star})=0$.
The second inequality uses the fact that for any $\xi\in\mathbb{R}^{d}$:
$\mathbb{E}\TNormS{\xi-\mathbb{E}\xi}=\mathbb{E}\TNormS{\xi}-\TNormS{\mathbb{E}\xi}\leq\mathbb{E}\TNormS{\xi}$.
In the last inequality we use Lemma \ref{lem:L2_grad}.
\end{proof}
Now we can proceed to prove the main theorem.
\begin{thm}
\label{thm:converge_proof} Suppose Assumptions \ref{assu:strong-convex-appendix},\ref{assu:grad-int-appendix},\ref{assu:finite-L_eta-appendix}
hold, and let $\x^{\star}=\arg\,\min_{\x}f(\x)$ and $L_{\sup}=\sup_{\veta\in\Omega}L_{\veta}$.
In addition, assume that there exists $0<\alpha<\frac{1}{2L_{\sup}}$
a sufficiently large $m$ such that

\[
\rho=\frac{1}{\gamma\alpha(1-2L_{\sup}\alpha)m}+\frac{2L_{\sup}\alpha}{(1-2L_{\sup}\alpha)}<1\,.
\]
Then SVRG (Algorithm \ref{alg:svrg_int}) with option II has geometric
convergence in expectation:

\[
\mathbb{E}\left[f(\tilde{\x}_{s})\right]-f(\x^{\star})\leq\rho^{s}\left(f(\tilde{\x}_{0})-f(\x^{\star})\right)\,.
\]
\end{thm}

\begin{proof}
From Assumption~\ref{assu:strong-convex-appendix}, and using Eq.~
(\ref{eq:vk_grad}) we have
\begin{align}
f(\x^{\star})-f(\x_{k-1}) & \geq-\nabla f(\x_{k-1})^{\T}(\x_{k-1}-\x^{\star})\label{eq:strong_conv_x_k}\\
f(\x_{k})-f(\x_{k-1}) & \geq-\alpha\mathbb{E}\left[\v_{k}\right]^{\T}\v_{k}\,.\nonumber 
\end{align}
Subtracting these inequalities yields
\begin{equation}
-\nabla f(\x_{k-1})^{\T}(\x_{k-1}-\x^{\star})\leq f(\x^{\star})-f(\x_{k})-\alpha\mathbb{E}\left[\v_{k}\right]^{\T}\v_{k}\,.\label{eq:grad_k-1}
\end{equation}

Thus, we have
\begin{align*}
\mathbb{E}\TNormS{\x_{k}-\x^{\star}} & =\mathbb{E}\TNormS{\x_{k-1}-\alpha\v_{k}-\x^{\star}}\\
 & =\TNormS{\x_{k-1}-\x^{\star}}-2\alpha\nabla f(\x_{k-1})^{\T}(\x_{k-1}-\x^{\star})+\alpha^{2}\mathbb{E}\TNormS{\v_{k}}\\
 & \leq\TNormS{\x_{k-1}-\x^{\star}}-2\alpha\left(f(\x_{k})-f(\x^{\star})\right)-2\alpha^{2}\mathbb{E}\left[\v_{k}\right]^{\T}\v_{k}\\
 & +4L_{\sup}\alpha^{2}\left(f(\x_{k-1})-2f(\x^{\star})+f(\tilde{\x})\right)
\end{align*}
where the equality uses Eq.~(\ref{eq:vk_grad}) and the inequality
uses Eq.~(\ref{eq:grad_k-1}) and Corollary \ref{cor:grad_differ}.

Now, consider a fixed stage $s$, such that $\x_{0}=\tilde{\x}=\tilde{\x}_{s-1}$
and $\tilde{\x}_{s}=\frac{1}{m}\sum_{k=1}^{m}\x_{k}$. By summing
the previous inequality over $k=1,\ldots,m$ and taking expectation
with respect to the history of the random variables $\eta_{1},\ldots,\eta_{m}$,
we obtain
\begin{align*}
\mathbb{E}\Vert\x_{m}-\x^{\star}\Vert^{2} & \leq\TNormS{\x_{0}-\x^{\star}}-2\alpha\sum_{k=1}^{m}\left(\mathbb{E}\left[f(\x_{k})\right]-f(\x^{\star})\right)-2\alpha^{2}\sum_{k=1}^{m}\TNormS{\mathbb{E}\left[\v_{k}\right]}\\
 & +4L_{\sup}\alpha^{2}\sum_{k=1}^{m}\left(\mathbb{E}\left[f(\x_{k-1})\right]-f(\x^{\star})\right)+4L_{\sup}\alpha^{2}m\left(f(\tilde{\x})-f(\x^{\star})\right)\\
 & \leq\TNormS{\tilde{\x}-\x^{\star}}-2\alpha\sum_{k=1}^{m}\left(\mathbb{E}\left[f(\x_{k})\right]-f(\x^{\star})\right)-2\alpha^{2}\sum_{k=1}^{m}\TNormS{\mathbb{E}\left[\v_{k}\right]}\\
 & +4L_{\sup}\alpha^{2}\sum_{k=1}^{m}\left(\mathbb{E}\left[f(\x_{k})\right]-f(\x^{\star})\right)+4L_{\sup}\alpha^{3}\sum_{k=1}^{m}\TNormS{\mathbb{E}\left[\v_{k}\right]}+4L_{\sup}\alpha^{2}m\left(f(\tilde{\x})-f(\x^{\star})\right)\\
 & \leq\TNormS{\tilde{\x}-\x^{\star}}-2\alpha\left(1-2L_{\sup}\alpha\right)\sum_{k=1}^{m}\left(\mathbb{E}\left[f(\x_{k})\right]-f(\x^{\star})\right)+4L_{\sup}\alpha^{2}m\left(f(\tilde{\x})-f(\x^{\star})\right)\\
 & \leq\left(\frac{2}{\gamma}+4L_{\sup}\alpha^{2}m\right)\left(f(\tilde{\x}_{s-1})-f(\x^{\star})\right)-2\alpha\left(1-2L_{\sup}\alpha\right)\sum_{k=1}^{m}\left(\mathbb{E}\left[f(\x_{k})\right]-f(\x^{\star})\right)\,.
\end{align*}

The second inequality is due to the strong convexity in Eq. (\ref{eq:strong_conv_x_k}),
and the third inequality uses the assumption $2L_{\sup}\alpha<1$
such that $(4L_{\sup}\alpha^{3}-2\alpha^{2})\TNormS{\mathbb{E}\left[\v_{k}\right]}\leq0$.
The last inequality uses Assumption \ref{assu:strong-convex-appendix}
with $\x$ replaced by $\tilde{\x}$ and $\y$ replaced by $\x^{\star}$.
In addition, $f(\tilde{\x}_{s})\leq\frac{1}{m}\sum_{k=1}^{m}f(\x_{k})$
due to the convexity of $f$. Therefore, we obtain 
\[
2\alpha\left(1-2L_{\sup}\alpha\right)m\left(\mathbb{E}\left[f(\tilde{\x}_{s})\right]-f(\x^{\star})\right)\leq\left(\frac{2}{\gamma}+4L_{\sup}\alpha^{2}m\right)\left(f(\tilde{\x}_{s-1})-f(\x^{\star})\right)\,.
\]
Dividing both sides of the above inequality by $2\alpha\left(1-2L_{\sup}\alpha\right)m$
gives
\[
\mathbb{E}\left[f(\tilde{\x}_{s})\right]-f(\x^{\star})\leq\rho^{s}\left(f(\tilde{\x}_{0})-f(\x^{\star})\right)\,.
\]

\end{proof}
 
\end{document}